\documentclass[11pt,oneside,reqno]{article}
\pdfoutput=1
\usepackage[utf8]{inputenc}
\usepackage[a4paper,vmargin=3cm]{geometry}
\usepackage{amsmath}
\usepackage{amssymb}
\usepackage{amsthm}
\usepackage{mathtools}
\usepackage{color}
\usepackage{stackrel}
\usepackage[activate={true,nocompatibility},final,tracking=true,
  kerning=true,spacing=true,factor=1100,stretch=10,shrink=10]{microtype}
\usepackage{hyperref}

\definecolor{Darkgreen}{rgb}{0,0.4,0}
\hypersetup{colorlinks,linkcolor=Darkgreen,citecolor=Darkgreen,breaklinks=true,unicode}

\microtypecontext{spacing=nonfrench}

\newtheorem{theorem}{Theorem}
\newtheorem{lemma}[theorem]{Lemma}
\newtheorem{proposition}[theorem]{Proposition}
\newtheorem{corollary}[theorem]{Corollary}

\theoremstyle{definition}
\newtheorem{remark}[theorem]{Remark}

\newtheorem*{acknowledgements}{Acknowledgements}

\newcommand{\treegraph}{{\mathbb T_d}}
\newcommand{\treegraphplus}{{\mathbb T^+_d}}

\newcommand{\rot}{{\textup{o}}}
\newcommand{\bbone}{\boldsymbol{1}}

\numberwithin{equation}{section}   
\numberwithin{theorem}{section}

\begin{document}

\title{\bf \large
  \uppercase{Level-set percolation of the {G}aussian free
  field on regular graphs {I}: regular
  trees}}

\date{Preliminary draft}

\author{Angelo Ab\"acherli
  \thanks{Departement Mathematik, ETH Z\"urich,
    R\"amistrasse 101, 8092 Z\"urich, Switzerland}
  \and Ji\v{r}\'i \v{C}ern\'y
  \thanks{Departement Mathematik und Informatik,
    University of Basel, Spiegelgasse 1, 4051 Basel, Switzerland}}

\maketitle

\begin{abstract}
  We study  level-set percolation of the Gaussian free field on the
  infinite \mbox{$d$-regular} tree for fixed $d\geq 3$. Denoting by
  $h_\star$ the critical value, we obtain the following results: for
  $h>h_\star$ we derive estimates on conditional exponential moments of
  the size of a fixed connected component of the level set above level $h$;
  for $h<h_\star$ we prove that the number of vertices connected over
  distance $k$ above level $h$ to a fixed vertex grows exponentially in
  $k$ with positive probability. Furthermore, we show that the
  percolation probability is a continuous function of the level $h$, at
  least away from the critical value $h_\star$. Along the way we also
  obtain matching upper and lower bounds on the eigenfunctions involved
  in the spectral characterisation of the critical value $h_\star$ and
  link the probability of  a non-vanishing  limit of the martingale used
  therein to the percolation probability. A number of the results derived
  here are applied in the accompanying  paper~\cite{AC2}.
\end{abstract}



\section{Introduction}
\label{s:intro}

In this article we investigate the Gaussian free field on $d$-regular
trees with $d\geq 3$. We focus in particular on the level sets  and their
behaviour in connection with level-set percolation. The goal is to obtain
a good description of the nature of the level sets for levels away from
the critical value of level-set percolation.

Level-set percolation of the Gaussian free field is a significant
representative of a percolation model with long-range dependencies and it
has attracted attention for a long time, dating back to \cite{31},
\cite{28} and \cite{12}. More recent developments can be found for
instance in \cite{7}, \cite{13}, \cite{22}, \cite{40}, \cite{DPR2},
\cite{Nitz} and \cite{ChiaNitz}. The particular case of Gaussian free
field on regular trees was studied before in \cite{35} and
\cite{SzniECP}, and on general transient trees in~\cite{37}. Compared to
the present article, the emphasis in these three papers is put on a
different aspect of the Gaussian free field, namely its connection with
the model of random interlacements.

Studying level-set percolation of the Gaussian free field on regular
trees specifically is of intrinsic interest. The case of the regular tree
comes along with strong tools based on the structure and symmetry of the
graph. These allow for often very exact computations which potentially
lead to especially explicit, though not at all trivial,  results. They
also make it one of the most promising setups for understanding level-set
percolation of the Gaussian free field  near criticality.

Besides the fact that the results obtained in this article are
interesting in their own right, our prime motivation comes from a concrete
application: in the accompanying  paper \cite{AC2} we prove a phase
transition in the behaviour of the level sets of the \emph{zero-average}
Gaussian free field on a certain class of \emph{finite} $d$-regular
graphs that are locally (almost) tree-like. This class includes $d$-regular
expanders of large girth and typical realisations of random $d$-regular
graphs. In a certain sense, the Gaussian free field on the $d$-regular
tree provides the  \emph{local picture} of the zero-average Gaussian free
field on these finite graphs, and its detailed understanding developed in
the present article is a key ingredient for~\cite{AC2}.

We now describe our results more precisely. Let $d\geq 3$ and denote by
$\treegraph$ the infinite $d$-regular tree. On $\treegraph$ we consider
the Gaussian free field with law $\mathbb P^\treegraph$ on
$\mathbb R^\treegraph$ and canonical coordinate process
$(\varphi_\treegraph(x))_{x\in \treegraph}$ so that,
\begin{equation}
  \label{0.2}
  \parbox{12cm}{
    under $\mathbb P^\treegraph$, $(\varphi_\treegraph(x))_{x \in \treegraph}$ is a
    centred Gaussian field on $\treegraph$ with covariance $\mathbb E^\treegraph
      [\varphi_\treegraph(x) \varphi_\treegraph(y)] = g_\treegraph(x,y)$ for all $x,y
      \in \treegraph$, where  $g_\treegraph(\cdot,\cdot)$ is the Green function of
    simple random walk on $\treegraph$ (see \eqref{1.1}).
  }
\end{equation}
Our main interest lies in investigating properties of the level sets of
$\varphi_\treegraph$, i.e.~of
\begin{equation*}
  E_{\varphi_\treegraph}^{\geq h}
  \coloneqq \{x\in \treegraph \, | \,
    \varphi_\treegraph(x)\geq h\}\quad\text{for } h\in \mathbb R.
\end{equation*}
In particular, we are interested in the connected component of
$E_{\varphi_\treegraph}^{\geq h}$ containing a fixed vertex
$\rot \in \treegraph$ (called \emph{root}) and denoted by
\begin{equation}
  \label{e:defCo}
  \mathcal C_{\rot}^{h} \coloneqq \big\{x\in \treegraph \, \big| \,
    x \text{ is connected to $\rot$ in $E_{\varphi_\treegraph}^{\geq h}$}\big\}.
\end{equation}

With this notation at hand we can define the critical value of level-set
percolation of the Gaussian free field via
\begin{equation}
  \label{0.5}
  h_\star \coloneqq \inf \big \{
    h\in \mathbb R \, \big| \,
    \mathbb P^\treegraph \big[ |\mathcal C_{\rot}^{h}|=\infty
      \big] = 0 \big \}.
\end{equation}
We point out that there is \emph{no} explicit formula for $h_\star$, even
though we consider the Gaussian free field on a regular tree. However,
the special structure of the underlying graph allows for a crucial
spectral characterisation of the critical value, as derived in \cite{35}.
We recall it in details in Section \ref{subsection1.1}. Very roughly, in
this spectral characterisation one associates to any level
$h\in \mathbb R$ a self-adjoint, non-negative operator $L_h$ on
$L^2(\mathbb R, \mathcal B(\mathbb R), \nu)$, where $\nu$ is a certain
centred Gaussian measure. The operator $L_h$ is naturally linked to the
distribution of $\varphi_\treegraph$ at a vertex conditioned on the value
of $\varphi_\treegraph$ at a neighbouring vertex and truncated below
level $h$ (see \eqref{e:defLh} and below it). One then considers the
operator norms $(\lambda_h)_{h\in \mathbb R}$ of the operators
$(L_h)_{h\in \mathbb R}$ and finds that (see  \cite{35}, Section 3)
\begin{equation}
  \label{briefspectralcharacterisation}
  \parbox{12cm}{
    the map $h \mapsto \lambda_h$ is a
    decreasing homeomorphism from $\mathbb R$ to $(0,d-1)$ and $h_\star$ is
    the unique value in $\mathbb R$ such that $\lambda_{h_\star} = 1$.
  }
\end{equation}
Additionally, for $h\in \mathbb R$ one has that $\lambda_h$ is a simple
eigenvalue of $L_h$ which is associated to a unique, non-negative
eigenfunction $\chi_h$ with unit $L^2$-norm, vanishing on $(-\infty, h)$
and positive elsewhere. So far one very important aspect of the
eigenfunctions $(\chi_h)_{h\in \mathbb R}$ was unknown, namely the
precise understanding of their asymptotic behaviour. Some care is applied
in \cite{35} to circumvent this lack of control (see  Remark 3.4 and
  Remark 4.4 therein).

As a first result in this paper, we close this gap and we obtain matching
upper and lower bounds on the eigenfunctions $(\chi_h)_{h\in \mathbb R}$.
In essence, we show in Proposition \ref{p:chibounds} that for every
$h\in \mathbb R$ there exist $c_h,c_h'>0$ such that
\begin{equation}
  \label{e:mainresultchi}
  0< c_h\, a^{1-\log_{d-1}(\lambda_h)}
  \le \chi_h(a) \le c'_h\, a^{1-\log_{d-1}(\lambda_h)}
  \quad \text{for all } a\in [h,\infty)
\end{equation}
(see also Remark \ref{r:slopeofchi} (i)). Presumably, such exact bounds
might be helpful when tackling level-set percolation questions of
$\varphi_\treegraph$ near the critical value $h_\star$. In this paper we
will use the upper bound to show the exponential growth of
$|\mathcal C_\rot^h|$ for $h<h_\star$ (see \eqref{e:preexpgrow}).

We also obtain another result related to the spectral characterisation of
level-set percolation of $\varphi_\treegraph$. It concerns the
non-negative martingale $(M_k^{\geq h})_{k\geq 0}$ for $h\in \mathbb R$
on which the proof  of the spectral characterisation of $h_\star$ in
\cite{35} heavily relies and in which the eigenfunction $\chi_h$ and the
associated eigenvalue $\lambda_h$ appear (see \eqref{209} for the
  definition). We show in the present paper  that  for all
$h\in \mathbb R \setminus \{h_\star\}$ the probability of a non-vanishing
martingale limit is equal to the `forward percolation probability' (see
  Proposition \ref{equivalence})
\begin{equation}
  \label{mainresultmartingale}
  \mathbb P^\treegraph \big[ |\mathcal C_\rot^h \cap
    \treegraphplus| = \infty  \big]
  = \mathbb P^\treegraph \big[ M^{\geq h}_\infty > 0\big],
\end{equation}
where $\treegraphplus \subseteq \treegraph$ is the `forward tree', that
is, the subtree of $\treegraph$ containing the root $\rot$ and in which
each vertex except for the root $\rot$ has $d$ neighbours and the root
$\rot$ has $d-1$ neighbours (the precise definition is given below
  \eqref{830}).

We moreover investigate the continuity properties in $h$ of percolation
probabilities like on the left hand side of \eqref{mainresultmartingale}
and as a third result we show in Theorem \ref{t:etacontinuity} that
\begin{equation}
  \label{e:defeta}
  \parbox{12cm}{
    the percolation probability $\eta (h) \coloneqq
    \mathbb P^\treegraph [|\mathcal C_{\rot}^{h}|=\infty]$ for
    $h\in \mathbb R$ and the forward percolation probability
    $\eta^+(h)\coloneqq \mathbb P^\treegraph \big[ |\mathcal C_\rot^h \cap
      \treegraphplus| = \infty  \big]$ for  $h\in \mathbb R$
    are continuous functions on $\mathbb R\setminus  \{h_\star\}$.}
\end{equation}

We then turn to  $\mathcal C_\rot^h$ and we obtain rather precise
estimates of its cardinality in both the subcritical $(h>h_\star)$ and
supercritical $(h<h_\star)$ phase.

We show that if $h>h_\star$, then there is some $\delta>0$ such that for
all $\gamma>0$ we can find $c_{h,\gamma},c'_{h,\gamma}>0$ satisfying (see
  Theorem~\ref{t:exponentialboundong})
\begin{equation}
  \label{e:mainresultsubcritical}
  \mathbb E^\treegraph \Big[ (1+\delta)^{| \mathcal C_{\rot}^{h}\cap \treegraphplus   |}
    \, \big| \, \varphi_\treegraph(\rot) = a \Big]
  \leq c_{h,\gamma} \exp(c'_{h,\gamma} a^{1+\gamma}) \quad \text{for all } a\geq h.
\end{equation}
In particular, this will imply that
$| \mathcal C_\rot^h \cap \treegraphplus |$ has exponential moments.
Incidentally, it also implies conditional exponential-tail estimates of
$| \mathcal C_{\rot}^{h}| $ of the form
\begin{equation}
  \label{e:preexptail}
  \mathbb P^\treegraph\big[|\mathcal C_{\rot}^{h}|\ge k
    \,\big |\,
    \varphi_\treegraph(\rot) = a\big]
  \le c_{h,a} e^{-c'_h k} \quad \text{for all } k \geq 1 \text{ and }  a\ge h,
\end{equation}
with a control of the dependence of the constant $c_{h,a}>0$ on the value
$a=\varphi_{\treegraph}(\rot)$ of the field at the root (see Remark
  \ref{r:linktointroduction}).

Finally, for $h<h_\star$ we prove that the number of vertices connected
over distance $k$ above level $h$ to the root $\rot \in \treegraph$ grows
exponentially  in $k$ with positive probability. This can be shown by
using our first result \eqref{e:mainresultchi} in combination with
\eqref{mainresultmartingale}. More precisely, with $S_\treegraph(\rot,k)$
denoting the sphere of radius $k\geq 0$ around $\rot$ in $\treegraph$, we
prove that (see Theorem \ref{t:expgrowth})
\begin{equation}
  \label{e:preexpgrow}
  \lim_{k \to \infty}
  \mathbb P^\treegraph \big[ |\mathcal C_{\rot}^{h} \cap \treegraphplus \cap
    S_\treegraph(\rot,k) |
    \geq \tfrac{\lambda_h^k}{k^2}  \big]  =  \eta^+(h) >0.
\end{equation}
We remind that here $\lambda_h>1$ is the eigenvalue from
\eqref{briefspectralcharacterisation}.

As explained earlier, we will see in the accompanying  paper \cite{AC2}
that the Gaussian free field $\varphi_\treegraph$ on $\treegraph$ in
essence plays the role of the local picture of the \emph{zero-average}
Gaussian free field on a specific class of \emph{finite} $d$-regular
graphs that are locally (almost) tree-like. By exploiting this feature,
we establish in \cite{AC2} a phase transition for level-set percolation
of the zero-average Gaussian free field on the finite graphs which is
characterised by the critical value $h_\star$ \emph{on the infinite
  tree}. Roughly, the strategy of \cite{AC2} is to use the local picture
to transfer the problem from the finite graphs to $\treegraph$ and then
to use the new results developed in the present article. Specifically, we
apply the estimate \eqref{e:mainresultsubcritical} in the proof of the
subcritical phase (\cite{AC2}, Theorem 3.1) and the two
results \eqref{e:defeta} and \eqref{e:preexpgrow} in the proof of the
supercritical phase (\cite{AC2}, Theorem 4.1). As an aside,
let us mention that for the application in \cite{AC2} it is crucial that
the exponent $\gamma$ on the right hand side of
\eqref{e:mainresultsubcritical} can be chosen strictly smaller than 2.

A similar approach as explained in the previous paragraph was carried out
in \cite{21} to describe a phase transition for the vacant set of simple
random walk on the same class of finite graphs as considered in the
accompanying  paper \cite{AC2}. As shown in \cite{21}, the local picture
in that case is given by the vacant set of random interlacements on
$\treegraph$. Thanks to the detailed understanding of random
interlacements in the infinite model, the phase transition in the finite
model can be established. An advantage of random interlacements on a tree
is that the connected components of its vacant set can be described
rather easily. Indeed, as observed in \cite{Tei09}, they are distributed
as Galton-Watson trees with a binomial offspring distribution. Thus,
properties like \eqref{e:preexptail} or \eqref{e:preexpgrow} are classical.

In contrast, the connected components  of the level sets of
$\varphi_\treegraph$, which play the corresponding role in our setup,
are not  Galton-Watson trees. The situation is more complicated and
obtaining results like \eqref{e:preexptail} and \eqref{e:preexpgrow} is
not straightforward. Instead, by the domain Markov property of the
Gaussian free field, one can view $\mathcal C_\rot^{h}$ for
$h\in \mathbb R$ as a certain multi-type branching process with an
uncountable type space. Some of the results in this paper are similar to
classical results about branching processes, though to our knowledge they
are not covered by the literature. We would like to stress that our
arguments rely on the \emph{special structure} of the Gaussian free field
on regular trees. Let us also mention that, despite the connection
between the Gaussian free field and random interlacements via isomorphism
theorems, we are not aware of any technique allowing to transfer the
results of \cite{Tei09} and \cite{21} directly to our context.

The structure of the article is as follows. In Section \ref{s:notation}
we collect the main part of the notation and some  known results about the
Gaussian free field on $\treegraph$. In particular, in Section
\ref{subsection1.1} we recall the spectral description of the critical
value $h_\star$ obtained in \cite{35}. In Section \ref{s:chibounds} we
derive asymptotic bounds on the eigenfunctions appearing in the spectral
description of $h_\star$. In Section \ref{s:recursiveequation} we give a
recursive equation for the conditional non-percolation probability.
Subsequently, we analyse the behaviour of the level sets of the Gaussian
free field in the supercritical phase in Section \ref{s:supercritical}.
This includes the continuity of the percolation probability (Section
  \ref{ss:continuity}) and the geometrical growth of level sets (Section
  \ref{ss:geometricgrowth}). Finally, in Section \ref{s:subcritical} we
investigate the subcritical phase and show that the cardinality of the
connected component of the level set containing the root has exponential
moments (and more).

\begin{acknowledgements}
  The authors wish to express their gratitude to A.-S.~Sznitman for
  suggesting the problem and for the valuable comments made at various
  stages of the project.
\end{acknowledgements}



\section{Notation and useful results}
\label{s:notation}

We start by  introducing the notation and recalling known properties of
the Green function and the Gaussian free field on $\treegraph$. These
include a recursive construction of $\varphi_\treegraph$
(Section~\ref{ss:recursive}) and the spectral characterisation of
$h_\star$ (Section \ref{subsection1.1}).

As mentioned earlier, we consider for fixed $d\geq 3$ the $d$-regular
tree $\treegraph$ with root $\rot$. We endow $\treegraph$ with the usual
graph distance $d_\treegraph(\cdot,\cdot)$. For any $R\geq0$ and
$x \in \treegraph$ we let
$B_\treegraph(x,R) \coloneqq \{y \in \treegraph \, | \, d_\treegraph(x,y) \leq R \}$
and
$S_\treegraph(x,R) \coloneqq \{y \in \treegraph \, | \, d_\treegraph(x,y) = R \}$
denote the ball and the sphere of radius $R$ around $x$, respectively.
For $x,z \in \treegraph$ a  path from $x$ to $z$ is a sequence of
vertices $x=y_0,y_1,\ldots,y_m=z$ in $\treegraph$ for some $m\geq 0$ such
that $y_i$ and $y_{i-1}$ are neighbours for all $i=1,\ldots,m$ (if
  $m\geq 1$). It is a \emph{geodesic} path from $x$ to $z$ if it is the
path of shortest length.

For $x\in \treegraph \setminus \{\rot\}$ let $\overline x$ be the unique
neighbour of $x$ on the geodesic path from $x$ to $\rot$. Moreover, let
$\overline\rot \in \treegraph$ denote an arbitrary fixed neighbour of
$\rot \in \treegraph$. For $x\in \treegraph$ we define
\begin{equation}
  \label{830}
  U_x  \coloneqq   \{  z \in \treegraph \, | \, \text{the geodesic
      path from $z$ to $x$ does not pass through $\overline x$} \}.
\end{equation}
In particular $\treegraph = \{\rot \} \cup \bigcup_{i=1}^d U_{x_i}$, where
$\{x_1,\ldots,x_d\}$ denote the neighbours of $\rot$. In the special case
of $x = \rot$ we write $\treegraphplus \coloneqq U_\rot$. We also set
$B_\treegraph^+(\rot,R)
\coloneqq \{{y \in \treegraphplus} \, | \, d_\treegraph(\rot,y) \leq R \}$
and similarly
$S_\treegraph^+(\rot,R)
\coloneqq \{y \in \treegraphplus \, | \, d_\treegraph(\rot,y) = R \}$
for $R\geq 0$.

We write $P_x^\treegraph$ for the canonical law of the simple random walk
$(X_k)_{k\ge 0}$ on $\treegraph$ starting at $x$ as well as
$E_x^\treegraph$ for the corresponding expectation.  Given
$U\subseteq \treegraph$ we write
$T_U \coloneqq \inf \{ k\geq 0 \, | \, X_k  \notin U\}$ for the exit time
from $U$ and $H_U \coloneqq \inf \{ {k\geq 0}\, | \, X_k \in U \}$ for
the entrance time in $U$ (here we set $\inf \emptyset \coloneqq \infty$).
In the special case of $U=\{z \}$ we use $H_z$ in place of $H_{\{z\}}$.
Recall that (see e.g.~\cite{33}, proof of Lemma 1.24)
\begin{equation}
  \label{e:hitproba}
  P_{x}^\treegraph \big[ H_{y} < \infty \big]
  = \Big(\frac{1}{d-1}\Big)^{d_\treegraph(x,y)}\quad \text{for } x,y\in \treegraph.
\end{equation}

The Green function $g_\treegraph(\cdot,\cdot)$ of simple random walk on
$\treegraph$ is given by (see \cite{33}, Lemma 1.24, for the explicit
  computation)
\begin{equation}
  \label{1.1}
  g_\treegraph(x,y)
  \coloneqq E_x^\treegraph \Big[\sum_{k=0}^\infty
    \bbone_{\{X_k=y\}} \Big]
  = \frac{d-1}{d-2} \Big( \frac1{d-1} \Big)^{d_\treegraph(x,y)}
  \quad\text{for } x,y
  \in \treegraph.
\end{equation}
For $U\subseteq \treegraph$ the Green function $g_\treegraph^U(\cdot,\cdot)$
of simple random walk on $\treegraph$ killed when exiting $U$ is
\begin{equation*}
  g_\treegraph^U(x,y) \coloneqq E_x^\treegraph \Big[\sum_{0\leq k < T_U}
    \bbone_{\{X_k=y\}} \Big]
  \quad \text{for } x,y \in \treegraph.
\end{equation*}
The functions $g_\treegraph(\cdot,\cdot)$ and $g^U_\treegraph(\cdot,\cdot)$
are symmetric and finite,  and $g^U_\treegraph(\cdot,\cdot)$
vanishes whenever $x\notin U$ or $y\notin U$. They are related by
the identity
\begin{equation}
  \label{1.4}
  g_\treegraph(x,y) = g_\treegraph^U(x,y) + E_x^\treegraph
  \big[g_\treegraph(X_{T_U},y) \bbone_{\{T_U < \infty\}}\big] \quad \text{for
  } x,y \in \treegraph,
\end{equation}
which is an easy consequence of the strong Markov property of simple
random walk at time $T_U$. In the particular case of $U \coloneqq U_x$
this implies that (by using \eqref{e:hitproba}, \eqref{1.1} and that
  $X_{T_U}=\overline x$ on $\{T_U<\infty\}$ under $P^\treegraph_x$)
\begin{equation}
  \label{e:gUxx}
  g_\treegraph^{U_{x}}(x,x)  = g_\treegraph(x,x) -
  \frac1{d-1} g_\treegraph(\overline x, x) =
  \frac{d-1}{d-2}\Big(1-\frac1{(d-1)^2}\Big)=\frac{d}{d-1}.
\end{equation}

Recall from \eqref{0.2} that $(\varphi_\treegraph(x))_{x \in \treegraph}$
is the centred Gaussian field with covariance given by
$g_\treegraph(\cdot,\cdot)$. It satisfies the following \emph{domain
  Markov property}: for $U \subseteq \treegraph$ define the new field
\begin{equation*}
  \varphi_\treegraph^U(x) \coloneqq  \varphi_\treegraph(x) - E_x^\treegraph[
    \varphi_\treegraph(X_{T_U}) \bbone_{\{T_U < \infty\}} ] \quad \text{for }
  x\in \treegraph.
\end{equation*}
Then,
\begin{equation}
  \label{1.18}
  \parbox{12cm}{
    under $\mathbb P^\treegraph$, $(\varphi_\treegraph^U(x))_{x\in \treegraph}$ is
    a centred Gaussian field on $\treegraph$ which is independent from
    $(\varphi_\treegraph(x))_{x\in \treegraph \setminus U}$ and has covariance
    $\mathbb E^\treegraph [\varphi_\treegraph^U(x) \varphi_\treegraph^U(y)] =
    g_\treegraph^U(x,y)$ for all $x,y \in \treegraph$.
  }
\end{equation}
The proof of this fact follows by an easy computation of covariances and
\eqref{1.4}.



\subsection{Recursive construction of the Gaussian free field}
\label{ss:recursive}

Property \eqref{1.18} can be applied to obtain a useful recursive
representation of the Gaussian free field on $\treegraph$ that we
introduce now. We point out that this description crucially relies on the
special features of the Gaussian free field when considered on a
(regular) tree.

Let $x\in \treegraph$ and let $\{x_1,\dots,x_I\}$ be the neighbours of $x$
not contained in the geodesic path from $x$ to $\rot$. In particular,
$I=d$ if $x=\rot$ and $I=d-1$ otherwise. We set
$U \coloneqq\bigcup_{i=1}^I U_{x_i}$. Since $\treegraph$ is a tree, it
can be easily seen  that
\begin{itemize}
  \item $g^U_\treegraph(x_i,x_i) = g^{U_{x_i}}_\treegraph(x_i,x_i)$ for
  $i\in \{1,\ldots,I\}$,

  \item $g^U_\treegraph(y,y')=0$ for $y\in U_{x_i}$, $y'\in U_{x_j}$
  where $i,j \in \{1,\ldots,I\}$ with $i\neq j$,

  \item for any $y\in U$, one has $X_{T_U}=x$ on $\{T_U<\infty\}$ under
  $P^\treegraph_y$.
\end{itemize}
Hence \eqref{1.18} together with \eqref{e:hitproba} and \eqref{e:gUxx}
yields that
\begin{equation}
  \label{e:recstep}
  \parbox{12cm}{under $\mathbb P^\treegraph$, conditionally on $\varphi_\treegraph (x)$,
    the random variables
    $(\varphi_\treegraph(x_i))_{1\le i \le I}$ are i.i.d.~Gaussians  with mean
    $\frac 1 {d-1}\varphi_\treegraph(x)$ and variance $\frac{d}{d-1}$.
    Furthermore, they are independent of $(\varphi_\treegraph(y))_{y \in \treegraph \setminus U}$.}
\end{equation}

Let now $(Y_x)_{x\in\treegraph}$ be a collection of independent centred
Gaussian variables defined on some auxiliary probability space
$(\Omega ,\mathcal A, \mathbb P)$ such that
$Y_\rot \sim \mathcal N(0,g_\treegraph(\rot,\rot)) =  \mathcal N(0,\frac{d-1}{d-2})$
and
$Y_x\sim \mathcal N(0,g_\treegraph^{U_x}(x,x)) = \mathcal N(0,  \frac{d}{d-1})$
for $x\neq \rot$. Define recursively
\begin{equation}
  \label{254}
  \widetilde \varphi (\rot) \coloneqq Y_\rot \quad \text{ and }
  \quad \widetilde \varphi(x) \coloneqq \frac 1 {d-1} \widetilde \varphi (\overline x)
  + Y_x \quad\text{for }x \in \treegraph \setminus \{ \rot \}.
\end{equation}
Then, by applying \eqref{e:recstep} iteratively, we see that
\begin{equation}
  \label{blau}
  \text{under $\mathbb P$, the law of
    $(\widetilde \varphi (x))_{x\in \treegraph}$  is $\mathbb P^\treegraph$,}
\end{equation}
so that \eqref{254} can serve as an  alternative construction of the
Gaussian free field on $\treegraph$.

The recursive representation \eqref{254}  has many useful consequences
and it will be used repeatedly throughout the paper. In particular, it
gives  a representation of the conditional distribution of
$\varphi_\treegraph$ given $\varphi_\treegraph(\rot)=a\in \mathbb R$,
\begin{equation}
  \label{e:defPa}
  \mathbb P_a^\treegraph
  \big[ (\varphi_\treegraph(y))_{y\in \treegraph} \in \cdot \,  \big]
  \coloneqq \mathbb P^\treegraph \big[
    (\varphi_\treegraph(y))_{y\in \treegraph} \in \cdot \ \big| \,
    \varphi_\treegraph(\rot)=a \big],
\end{equation}
with corresponding expectation $\mathbb E_a^\treegraph$. Moreover, if we
let $x_1,\dots,x_d$ denote the neighbours of the root
$\rot \in \treegraph$, then from \eqref{254} and \eqref{blau} it follows
that for every $a\in \mathbb R$,
\begin{equation}
  \label{275}
  \parbox{12cm}{under $\mathbb P_a^\treegraph$, the random  fields
    $(\varphi_\treegraph(y))_{y\in U_{x_i}}$ for $i=1,\dots,d$ are
    independent. Furthermore, for any  event
    $A \in \sigma(\varphi_\treegraph(z), z\in \treegraphplus)$ and
    $i=1,\dots,d$  one has
    $\mathbb P_a^\treegraph \big[ (\varphi_\treegraph(y))_{y\in U_{x_i}} \in A \big]
    = \mathbb E^Y \Big[  \mathbb P_{\frac{a}{d-1}+ Y}^\treegraph
      \big[ (\varphi_\treegraph(y))_{y\in \treegraphplus}  \in  A \big]
      \Big]$,
    where $Y \sim \mathcal N(0,\tfrac{d}{d-1})$ and $\mathbb E^Y$ is the expectation
    with respect to $Y$. }
\end{equation}
(In the equality in \eqref{275} we also use that the law of
  $\varphi_\treegraph$ on $U_{x_i}$ equals the law of $\varphi_\treegraph$
  on $\treegraphplus$.)

Due to \eqref{254} and \eqref{blau}, the Gaussian free field on
$\treegraph$ can be related to a multi-type branching process with type
space $\mathbb R$. Indeed, we can view every $x\in S_\treegraph(\rot,k)$
as an individual in the $k$-th generation of the branching process with
an attached type $\varphi_\treegraph(x)\in \mathbb R$. In this
perspective \eqref{e:recstep} can be rephrased as: every individual $x$
has $d-1$ children ($d$ children if $x=\rot$) whose types, conditionally
on $\varphi_\treegraph(x)$, are chosen independently according to the
distribution $\mathcal N(\frac 1{d-1}\varphi_\treegraph(x),\frac{d}{d-1})$.

This point of view can easily be adapted to $\mathcal C_\rot^h$ from
\eqref{e:defCo} as well, namely by considering the same multi-type
branching process but instantly killing all individuals whose type does
not exceed $h$. In other words, $\mathcal C_\rot^h$ can be viewed as a
multi-type branching process with a barrier and the percolation of
$\mathcal C_\rot^h$ corresponds to the non-extinction of this branching
process. However, while some of the results in this paper  are
reminiscent of classical results about branching processes, we would like
to emphasise that the proofs make heavy use of the special structure of
the Gaussian free field on a regular tree. We are going to recall one of
the special features in the next section.



\subsection{Spectral characterisation of the critical value}
\label{subsection1.1}

We now recall the spectral characterisation of the critical value
$h_\star$ from \cite{35}, which is central for our paper. Note that our
$d$-regular tree $\treegraph$ corresponds in the notation of \cite{35} to
the $(\widetilde d+1)$-regular tree $T$ with $\widetilde d \coloneqq d-1$.
Moreover in \cite{35}, in the definition of the Green function
$g_\treegraph(\cdot,\cdot)$ on the tree, there is an extra normalising
factor equal to the degree of the tree (see \cite{35}, (3.1)). This
explains the differences between the formulas to come  and the formulas
in \cite{35}.

Let $\nu = \mathcal N(0,\frac{d-1}{d-2})$ be the centred Gaussian measure
with variance
$\frac{d-1}{d-2} \overset{\eqref{1.1}}{=} g_\treegraph(\rot,\rot)$. For
$h\in \mathbb R$ define the operator
\begin{equation}
  \label{e:defLh}
  \begin{split}
    &(L_h f) (a) \coloneqq (d-1) \bbone_{[h,\infty)}(a) \,
    \mathbb E^Y \big[ f (\tfrac{a}{d-1} + Y ) \bbone_{[h,\infty)}
      (\tfrac{a}{d-1} + Y )    \big]
    \\ &\text{for } f \in L^2(\nu)\coloneqq L^2(\mathbb R, \mathcal  B(\mathbb R),\nu)
    \text{ and } a \in \mathbb R,
  \end{split}
\end{equation}
where $Y \sim \mathcal N(0,\tfrac{d}{d-1})$ and $\mathbb E^Y$ is the
expectation with respect to $Y$. The operator $L_h$ is closely linked to
the Gaussian free field and its level set above level $h$. Indeed, one
has
$(L_h f) (a) = \mathbb E_a^\treegraph \big[\sum_{x \in \mathcal C_{\rot}^{h}
    \cap S_\treegraph^+(\rot,1)} f(\varphi_\treegraph(x))\big]$
for  $a\geq h$  by \eqref{254}.

The following proposition summarises the known properties of the
operators $(L_h)_{h\in \mathbb R}$ and characterises the critical value
$h_\star$.

\begin{proposition}[\cite{35}, Propositions~3.1 and~3.3]
  \label{p:Lhprops}
  For every $h \in \mathbb R$ the operator $L_h$ is self-adjoint,
  non-negative and its operator norm
  \begin{equation}
    \label{17}
    \lambda_h \coloneqq \| L_h \|_{L^2(\nu) \to L^2(\nu)}
  \end{equation}
  is a simple eigenvalue of $L_h$.  Moreover, there is a unique,
  non-negative eigenfunction $\chi_h \in L^2(\nu)$ of $L_h$,
  corresponding to the eigenvalue $\lambda_h$, with
  $\| \chi_h \|_{L^2(\nu)}=1$.  The function $\chi_h$ is continuous and
  positive on $[h,\infty)$, and vanishing on $(-\infty,h)$. Additionally,
  the map $h \mapsto \lambda_h$ is a decreasing homeomorphism from
  $\mathbb R$ to $(0,d-1)$ and $h_\star$ is the unique value in
  $\mathbb R$ such that $\lambda_{h_\star} = 1$.
\end{proposition}

In Proposition~\ref{p:chibounds} in Section \ref{s:chibounds} we will
give matching upper and lower bounds on the eigenfunctions $\chi_h$.

On the way, we recall the following hypercontractivity estimate which is
a direct consequence of the hypercontractivity property of the
Ornstein-Uhlenbeck semigroup (see \cite{35}, (3.14)): for  $1<p<\infty$
and $q=(p-1)(d-1)^2+1$ one has (with $Y\sim \mathcal N(0,\frac{d}{d-1})$)
\begin{equation}
  \label{405}
  \Big\|  \mathbb E^Y \big[ f (\tfrac{\cdot}{d-1} + Y ) \big] \Big\|
  _{L^q(\nu)} \leq \| f \|_{L^p(\nu)} \quad \text{for $f\in L^p(\nu )$}.
\end{equation}
We will use the estimate \eqref{405}, with its precise relation between
the parameters $p$ and $q$, several times. Especially, it will be applied
to prove Proposition~\ref{1007} which computes the Fr\'echet derivative
of a certain operator. This will be a key ingredient for showing the
existence of conditional exponential moments of $|\mathcal C_\rot^h |$ in
the subcritical phase in Section~\ref{s:subcritical}.

Furthermore, for every $h\in \mathbb R$ there is a martingale
$(M^{\geq h}_k)_{k\geq0}$ closely related to $L_h$. Indeed, if we set
\begin{equation}
  \label{300}
  \mathcal Z_k^h \coloneqq \mathcal C_{\rot}^{h} \cap
  S_\treegraph^+(\rot,k) \quad \text{for } k\geq 0,
\end{equation}
then  (see \cite{35}, (3.31) and (3.35))
\begin{equation}
  \label{209}
  M^{\geq h}_k \coloneqq \lambda_h^{-k} \sum_{y\in  \mathcal Z_k^h}
  \chi_h(\varphi_\treegraph(y))  \quad \text{for }k\geq 0
\end{equation}
defines a non-negative martingale under $\mathbb P^\treegraph$ with
respect to the filtration $(\mathcal F_k)_{k\geq 0}$ given by
\begin{equation}
  \label{e:defFk}
  \mathcal F_k \coloneqq \sigma( \varphi_\treegraph(y) , y\in
    B_\treegraph^+(\rot,k)).
\end{equation}
In particular, $M^{\geq h}_k$ converges $\mathbb P^\treegraph$-almost
surely to some $M^{\geq h}_\infty \geq 0$ as $k\to \infty$ and (see
  \cite{35}, proof of Proposition 3.3)
\begin{equation}
  \label{278}
  \text{for $h<h_\star$ one has $\mathbb P^\treegraph[ M^{\geq h}_\infty > 0]
    >0$.}
\end{equation}
Note that there is a direct relation between the probability of a
non-vanishing martingale limit $M_\infty^{\geq h}$ and the forward
percolation probability
\begin{equation}
  \label{251}
  \eta^+(h) =  \mathbb P^\treegraph \big[ |\mathcal C_\rot^h \cap
    \treegraphplus| = \infty  \big]  \quad \text{for $h \in \mathbb R$}
\end{equation}
from \eqref{e:defeta}.
We only need to observe that
\begin{equation*}
  \{ |\mathcal C_\rot^h \cap \treegraphplus| <\infty \} \subseteq \{\mathcal
    Z_k^h = \emptyset \text{ for $k$ large enough}\} \subseteq
  \{M_\infty^{\geq h}=0\}.
\end{equation*}
Therefore $\mathbb P^\treegraph[ M^{\geq h}_\infty > 0] \leq \eta^+(h)$.
In Section~\ref{s:supercritical}  we will see that this inequality is
actually an equality, at least when $h\neq h_\star$
(Proposition~\ref{equivalence}). As a last observation, note that by a
union bound and the symmetry of $\treegraph$ we obtain that
$\eta^+(h) \leq \eta (h) \leq d\cdot \eta^+(h)$ for the percolation
probability $\eta$ from \eqref{e:defeta}. Hence by \eqref{0.5} one has
$\eta^+(h) = 0$ for $h>h_\star$ and $\eta^+(h) > 0$ for $h<h_\star$.

A final word on the convention followed concerning constants: by
$c,c',\ldots$ we denote  positive constants with values changing from
place to place and which only depend on the dimension $d$. The dependence
of constants on additional parameters  appears in the notation.



\section{Asymptotic behaviour of the eigenfunctions}
\label{s:chibounds}

The main result of this section are the matching bounds on the
eigenfunctions $(\chi_h)_{h\in \mathbb R}$ from
Proposition~\ref{p:Lhprops} collected in Proposition \ref{p:chibounds}
below (corresponding to \eqref{e:mainresultchi}). The upper bound will be
used later to show that connected components of supercritical level sets
grow exponentially with positive probability (Theorem~\ref{t:expgrowth}
  in Section~\ref{s:supercritical}). The corresponding lower bound is not
used further but it is included for completeness.

\begin{proposition}
  \label{p:chibounds}
  \textup{(i)} There exists $c>0$ (see in \eqref{205} below) such that
  for all $h \in \mathbb R$ one has
  \begin{equation}
    \label{e:upperboundchi}
    \chi_h(a) \leq c \,  a^{1-\log_{d-1}(\lambda_h)}
    \quad \text{for all }  a \geq d-1.
  \end{equation}

  \noindent \textup{(ii)} For every $h\in \mathbb R$ there exists
  $c_h>0$ such that
  \begin{equation}
    \label{e:lowerboundchi}
    \chi_h(a) \geq  c_h
    a^{1-\log_{d-1}(\lambda_h)}
    \quad \text{for all } a \geq h.
  \end{equation}
\end{proposition}

\begin{remark}
  \label{r:slopeofchi}
  \textup{(i)} From Proposition~\ref{p:Lhprops} recall that $\chi_h$ is
  continuous and strictly positive on $[h,\infty)$. Therefore, by
  adjusting the constant $c$, Proposition~\ref{p:chibounds} implies
  \eqref{e:mainresultchi}.

  \noindent \textup{(ii)} By Proposition~\ref{p:Lhprops} one has
  $\lambda_h\in (0,d-1)$. Hence, the exponent
  $\kappa_h:= 1-\log_{d-1}(\lambda_h)$ in \eqref{e:upperboundchi} and
  \eqref{e:lowerboundchi} is positive for all $h \in \mathbb R$.
  Moreover, $\kappa_h\in (0,1)$  for $h<h_\star$, $\kappa_h= 1$ for
  $h=h_\star$ and $\kappa_h>1$  for $h>h_\star$. \qed
\end{remark}

\begin{proof}[Proof of Proposition~\ref{p:chibounds}]
  (i) Let $(Y_i)_{i\geq 1}$ be i.i.d.~random variables with distribution
  $\mathcal N(0,\tfrac{d}{d-1})$. By iteratively using \eqref{e:defLh}
  and the fact that $\chi_h$ is the eigenfunction of $L_h$ with
  eigenvalue $\lambda_h$, we obtain  for every $a\in \mathbb R$ and
  $k\geq 1$
  \begin{align}
      \chi_h(a) &
      = \frac1{\lambda_h} (L_h \chi_h)(a)
      \leq \frac{d-1}{\lambda_h} \,
      \mathbb E^{Y_1} \Big[ \chi_h \big(\tfrac{a}{d-1} + Y_1 \big) \Big]
      = \frac{d-1}{\lambda_h}  \, \mathbb E^{Y_1} \Big[  \frac1{\lambda_h}
        (L_h \chi_h)\big(\tfrac{a}{d-1} + Y_1 \big) \Big]   \nonumber
      \\ & \leq \Big(\frac{d-1}{\lambda_h} \Big)^2 \,
      \mathbb E^{Y_1} \Big[  \mathbb E^{Y_2}
        \big[\chi_h \big(\tfrac{a}{(d-1)^2} + \tfrac1{d-1}Y_1 + Y_2 \big) \big]
        \Big]
      \label{200}
      \\ & \leq \ldots
      \leq \Big(\frac{d-1}{\lambda_h} \Big)^k  \, \mathbb E \Big[ \chi_h
        \big(\tfrac{a}{(d-1)^k} + \tfrac1{(d-1)^{k-1}}Y_1 +\ldots +
          \tfrac1{d-1}Y_{k-1}+ Y_k \big) \Big],   \nonumber
  \end{align}
  where the expectation on the right hand side of \eqref{200} is taken
  with respect to $Y_1,\ldots,Y_k$. Note that for any $k\geq 1$ the
  random variable
  $\tfrac1{(d-1)^{k-1}}Y_1 +\dots + \tfrac1{d-1}Y_{k-1}+ Y_k$ appearing
  on the right hand side of \eqref{200} is centred Gaussian with variance
  \begin{equation}
    \label{201}
    \sigma^2_k \coloneqq \frac{d}{d-1} \sum_{i=0}^{k-1}\frac1{(d-1)^{2i}} =
    \frac{d-1}{d-2}\Big(1-\frac1{(d-1)^{2k}}\Big) \leq \frac{d-1}{d-2} \eqqcolon
    \sigma^2.
  \end{equation}
  Hence, if we denote by $f_{\mu,\tau^2}$ the density of the normal
  distribution $\mathcal N(\mu,\tau^2)$ and
  $a_k \coloneqq \tfrac{a}{(d-1)^k}$ for $k\geq 1$, then (recall from
    above \eqref{e:defLh} that $\nu = \mathcal N(0,\sigma^2)$)
  \begin{equation}
    \label{202}
    \begin{split}
      \mathbb E \Big[  & \chi_h \big(\tfrac{a}{(d-1)^k} + \tfrac1{(d-1)^{k-1}}Y_1
        +\ldots + \tfrac1{d-1}Y_{k-1}+ Y_k \big)  \Big]  =  \int_{\mathbb R} \chi_h (y)
      f_{a_k,\sigma^2_k}(y) \, dy  \\
      &=  \int_{\mathbb R} \chi_h(y) \frac{ f_{a_k,\sigma^2_k}
        (y)}{f_{0,\sigma^2}(y)} \, \nu(dy)
      \overset{(*)}{\leq} \underbrace{\| \chi_h
      \|_{L^2(\nu)}}_{= 1}    \Bigg( \int_{\mathbb R}
      \frac{ f^2_{a_k,\sigma^2_k} (y)}{f_{0,\sigma^2}(y)} \, dy \Bigg)^{\frac12}  ,
    \end{split}
  \end{equation}
  where in $(*)$ we apply the Cauchy-Schwarz inequality.
  Note that for all $k\geq 1$
  \begin{equation}
    \label{203}
    \frac{ f^2_{a_k,\sigma^2_k} (y)}{f_{0,\sigma^2}(y)} =
    \frac{\sqrt{\sigma^2}}{\sqrt{2\pi}\sigma^2_k}
    \exp\Big(-\frac{(y-a_k)^2}{\sigma^2_k} + \frac{y^2}{2\sigma^2}\Big) \leq
    f_{2a_k,\sigma^2}(y)  \cdot \frac{\sigma^2}{\sigma^2_1}
    \exp\Big(\frac{a_k^2}{\sigma^2}\Big),
  \end{equation}
  where we use $\sigma^2_1 \leq \sigma^2_k \leq \sigma^2$ (see
    \eqref{201}). By combining \eqref{200} with \eqref{202} and
  \eqref{203} we obtain for any $a \in \mathbb R$ and $k\geq 1$
  \begin{equation}
    \label{204}
    \chi_h(a) \leq \Big(\frac{d-1}{\lambda_h} \Big)^k
    \sqrt{\frac{\sigma^2}{\sigma^2_1}} \exp\Big(\frac{a_k^2}{2\sigma^2}\Big).
  \end{equation}
  If $a \geq d-1$, we can apply \eqref{204} for
  $k = k(a) \coloneqq \lfloor \log_{d-1}(a) \rfloor$, that is, for the
  unique $k\geq 1$ with $(d-1)^k \leq a < (d-1)^{k+1}$. Since
  $\frac{d-1}{\lambda_h}>1$ by Proposition~\ref{p:Lhprops} and
  $a_{k(a)} \leq d-1$, we obtain from \eqref{204}
  \begin{equation}
    \label{205}
    \chi_h(a) \leq \Big(\frac{d-1}{\lambda_h} \Big)^{\log_{d-1}(a)}
    \underbrace{\sqrt{\frac{\sigma^2}{\sigma^2_1}}
      \exp\bigg(\frac{(d-1)^2}{2\sigma^2}\bigg)}_{\eqqcolon \, c >0}
      =  c \, a^{1-\log_{d-1}(\lambda_h)},
  \end{equation}
  which concludes the proof of part (i).

  (ii)  Let $K=K_h\geq 0$ be the smallest non-negative integer such that
  $(d-1)^{K}-14 \geq h$. Define the intervals
  \begin{equation}
    \label{276}
    \begin{split}
      &I_0 \coloneqq \big[h  , (d-1)^{K+1} +
        4\cdot 2 \big], \\
      &I_k \coloneqq \big[(d-1)^{K+k} - 4(k+2),
        (d-1)^{K+k+1} + 4 (k+2) \big] \quad \text{for $k\geq 1$},
    \end{split}
  \end{equation}
  which form a non-disjoint decomposition of $[h,\infty)$. (To see that
    all left boundaries of the intervals are larger than $h$ use the
    assumption on $K$ as well as $d\geq 3$.) The intervals are such that
  \begin{equation}
    \label{e:Iks}
    \frac{a}{d-1}- (k+1) \in I_{k-1} \quad
    \text{and} \quad \frac{a}{d-1} + (k+1) \in I_{k-1} \quad \text{ for $a\in I_k$, $k\geq 1$.}
  \end{equation}
  Indeed, for $k\geq 1$ and $d\geq 3$ one has
  $4\frac {k+2}{d-1} + (k+1) \leq 4(k+1)$  and hence for every~$a\in I_k$
  \begin{equation*}
    \begin{split}
      \tfrac{a}{d-1} +  (k+1) &\leq (d-1)^{K+k}+4\tfrac {k+2}{d-1} +
      (k+1) \leq (d-1)^{K+k}+4(k+1),  \\
      \tfrac{a}{d-1} -  (k+1) &\geq (d-1)^{K+k-1}-4\tfrac {k+2}{d-1} -
      (k+1) \geq  (d-1)^{K+k-1} - 4(k+1).
    \end{split}
  \end{equation*}

  By using once more that $\chi_h$ is the eigenfunction of $L_h$, we have
  for every $a\in I_{k}$ with $k\geq 1$ that
  \begin{equation}
  \label{something}
    \chi_h(a) = \frac1{\lambda_h} (L_h \chi_h)(a)
    \stackrel{\eqref{e:defLh}}{\geq}  \frac{d-1}{\lambda_h}
    \underbrace{\bbone_{[h,\infty)}(a)}_{=1} \,
    \mathbb E^{Y} \Big[ \chi_h
      \big(\tfrac{a}{d-1} + Y \big)
      \bbone_{\{|Y| \leq k+1\}} \Big].
  \end{equation}
  From \eqref{e:Iks} we have that on $\{ |Y|\le  k+1 \}$ it holds
  $\frac{a}{d-1}+Y\in I_{k-1}$. Hence,
  \begin{equation}
    \label{206}
    \chi_h(a)
     \stackrel{\eqref{something}}{\geq} \frac{d-1}{\lambda_h}
    \inf_{b\in I_{k-1}} \chi_h(b)
    \cdot \mathbb P \big[ |Y| \leq k+1 \big] \quad \text{for $a\in I_k$, $k\geq 1$}.
  \end{equation}
  The repeated application of \eqref{206} implies that
  \begin{equation}
    \label{207}
    \begin{split}
      \chi_h(a)
      &\geq \Big(\frac{d-1}{\lambda_h}\Big)^k
      \inf_{b\in I_0} \chi_h(b)
      \prod_{\ell=1}^k \big(1-\mathbb P[|Y|> \ell+1]\big)
     \quad \text{for $a\in I_k$, $k\geq 1$}.
    \end{split}
  \end{equation}
  By the exponential Markov inequality
  \begin{equation}
    \label{e:tailprod}
    \prod_{\ell=1}^k \big(1-\mathbb P[|Y|> \ell +1]\big)
    \ge \prod_{\ell=1}^k \big(1-2 e^{-\frac{d-1}{2d}(\ell+1)^2}\big)
    \overset{d\geq 3}{\geq}
    \prod_{\ell=1}^\infty \big(1-2 e^{-\frac{1}{3}(\ell+1)^2}\big)  \ge c
  \end{equation}
  with $c \in (0,1)$ independent of $k\ge 1$.

  The inequality \eqref{e:lowerboundchi} is now an easy consequence of
  \eqref{207} and \eqref{e:tailprod}. First note that
  $\chi_h(a) \geq \inf_{b\in I_0} \chi_h(b) > 0$ for
  $a \in [h,(d-1)^{K+1})$ since $\chi_h$ is continuous and strictly
  positive on $[h,\infty)$ by Proposition~\ref{p:Lhprops}, and
  $I_0 \subseteq [h,\infty)$ is compact. Now, for $a \geq (d-1)^{K+1}$
  let $k(a) \coloneqq \lfloor \log_{d-1}(a)-K \rfloor\geq 1$ be the
  unique integer with $(d-1)^{K + k(a)} \leq a < (d-1)^{K+k(a)+1}$. In
  particular $a \in I_{k(a)}$  and therefore  from \eqref{207} and
  \eqref{e:tailprod} we obtain that
  \begin{equation*}
      \chi_h(a)
      \geq  c\Big(\frac{d-1}{\lambda_h}\Big)^{k(a)} \inf_{b\in I_0} \chi_h(b)
      \ge \underbrace{c\Big(\frac{\lambda_h}{d-1}\Big)^{K+1}
        \inf_{b\in I_0}\chi_h(b)}_{\eqqcolon \, c_{h} >0} \,
      a^{\log_{d-1}(\frac{d-1}{\lambda_h}) }.
  \end{equation*}
  This shows \eqref{e:lowerboundchi} and concludes the proof of
  Proposition~\ref{p:chibounds}.
\end{proof}



\section{Recursive equation for the non-percolation probability}
\label{s:recursiveequation}

In this section we adopt the perspective of multi-type branching
processes and show that a certain function (see \eqref{250}) closely
related to the forward percolation probability from \eqref{251} is the
unique solution to a recursive equation (Theorem \ref{t:uniqueness}).
This fact will be used to derive the results on the supercritical
behaviour of the level sets of $\varphi_\treegraph$ in
Section~\ref{s:supercritical}, in particular the continuity of the
percolation probability and its equality with the probability of a
non-vanishing martingale limit (Theorem \ref{t:etacontinuity} and
  Proposition \ref{equivalence}). At the end of the section we compute
the Fr\'echet derivative of the operator involved in the recursive
equation (Proposition \ref{1007}). This will be an important ingredient
to estimate exponential moments of $|\mathcal C_\rot^h|$  in Section
\ref{s:subcritical} (Theorem \ref{t:exponentialboundong}). The proofs in
this section partially follow the lines from  multi-type branching
processes (see e.g.~\cite{Har63}, Chapter III). However, a lot depends on
the special structure of the Gaussian free field on regular trees.

We introduce for every $h\in \mathbb R$ the conditional forward
extinction probability (see \eqref{e:defPa} for notation)
\begin{equation}
  \label{250}
  q_h(a) \coloneqq  \mathbb P_a^\treegraph \big[
    |\mathcal C_\rot^h \cap \treegraphplus| < \infty
    \big]  \quad \text{for $a \in \mathbb R$}.
\end{equation}
The function $q_h$ is closely related to the value of the forward
percolation probability from \eqref{251} at $h$. Indeed,  since the
distribution $\nu$ above \eqref{e:defLh} is the distribution of
$\varphi_\treegraph(\rot)$ under $\mathbb P^\treegraph$, one has
\begin{equation}
  \label{280}
  \int_{\mathbb R} q_h(a) \, d\nu(a) =
  \mathbb E^\treegraph[q_h(\varphi_\treegraph(\rot))]
  \stackrel{\eqref{251}}{=} 1- \eta^+(h).
\end{equation}
In particular, this shows by the comment at the end of Section
\ref{s:notation} that $q_h$ is identically 1 if $h>h_\star$ and is not
identically 1 if $h<h_\star$.

Now, recall the space $L^2(\nu)$ defined in~\eqref{e:defLh} and let us
define for every $h \in \mathbb R$ the (non-linear) operator $R_h$ on
$L^2(\nu)$ through:
\begin{equation}
  \label{257}
  \begin{split}
    &(R_h f) (a) \coloneqq \bbone_{(-\infty,h)}(a)  +
    \bbone_{[h,\infty)}(a)  \mathbb E^Y \big[ f (\tfrac{a}{d-1} + Y )
      \big]^{d-1}  \\
    &\text{for } f \in L^2(\nu) \text{ and } a \in \mathbb R,
  \end{split}
\end{equation}
where $Y \sim \mathcal N(0,\tfrac{d}{d-1})$ as in \eqref{e:defLh}. To see
that indeed $R_h f \in L^2(\nu)$, abbreviate
$\hat f(a) \coloneqq \mathbb E^Y [ f (\tfrac{a}{d-1} + Y )]$ for
$a\in \mathbb R$ and apply the hypercontractivity estimate \eqref{405}
with $p= 2$ and $q= (d-1)^2+1$ to find that
$\| \hat f^{d-1} \|_{L^{q/d-1} (\nu)}
= \| \hat f \|^{d-1}_{L^{q} (\nu)}\leq \| f \|^{d-1}_{L^{2} (\nu)} < \infty$
and thus $\hat f^{d-1} \in L^{\frac{q}{d-1}}(\nu)$. Since
$\tfrac{q}{d-1} = (d-1)+\tfrac1{d-1} \geq 2$, this  implies that
$R_h f \in L^2(\nu)$.

We are actually only interested in the operator $R_h$ for
$h\in \mathbb R$  on the subset
\begin{equation*}
  \mathcal S_h \coloneqq \{ f\in L^2(\nu ) \, | \, 0\le f \le 1 \text{ and }
    f = 1 \text{ on } (-\infty,h) \}.
\end{equation*}
By definition we directly have $R_h: \mathcal S_h \to \mathcal S_h$. In
Theorem \ref{t:uniqueness} we prove that $q_h$ is essentially the unique
solution in $\mathcal S_h$ to the equation $R_h f = f$.

From the  multi-type branching process perspective, the operator $R_h$
can be used to write recurrence relations for generating functionals
related to $\mathcal  C_\rot^h$, cf.~for example  \cite{Har63},
Section~III.7. In particular, by using the notation from \eqref{e:defPa},
\eqref{300} and by applying~\eqref{e:recstep}, we see that
\begin{equation*}
  (R_h f) (a) = \mathbb E^\treegraph_a\Big[
    \prod_{y\in \mathcal Z_1^h}
    f(\varphi_\treegraph(y))\Big]
  \quad  \text{for } f\in \mathcal S_h \text{ and } a\in \mathbb R,
\end{equation*}
where the empty product is interpreted as being equal to 1.
This identity can be extended: define iteratively
\begin{equation*}
  \begin{split}
    R_h^0 f \coloneqq f \quad \text{and} \quad
    R_h^k f \coloneqq R_h^{k-1} (R_h f) \quad \text{for } f \in L^2(\nu) \text{
      and } k \geq 1.
  \end{split}
\end{equation*}
Then one can  prove  by induction on $k\geq 0$ and using \eqref{e:recstep} that
\begin{equation}
  \label{258}
  (R_h^k f)(a) = \mathbb E_a^\treegraph \Big[ \prod_{y\in \mathcal Z_k^h}
    f(\varphi_\treegraph(y)) \Big]
  \quad \text{for }  f \in \mathcal S_h, k\geq 0, a\in \mathbb R.
\end{equation}

We come to the main result of this section.

\begin{theorem}
  \label{t:uniqueness}
  For every $h\in \mathbb R$ the function $q_h$ is the smallest solution
  in $\mathcal S_h$ to the equation $f= R_h f$. More precisely, the only
  solutions in $\mathcal S_h$ to $R_h f = f$ are the function $q_h$ and
  the constant $1$ function. These two functions coincide if $h>h_\star$
  and are distinct if $h<h_\star$.
\end{theorem}

The proof of the theorem is broken into several steps stated as
Lemmas~\ref{l:receqn}--\ref{l:uniquesolution}. The first one  is a
classical observation from the theory of multi-type branching processes.
\begin{lemma}
  \label{l:receqn}
  Let $h\in \mathbb R$. The function $q_h$ satisfies $q_h\in \mathcal S_h$
  and solves the equation $R_h f = f$.
\end{lemma}
\begin{proof}
  The fact that $q_h\in \mathcal S_h$ is clear. To prove the second
  statement, denote
  $S_\treegraph^+(\rot,1) \eqqcolon \{x_1,\dots,x_{d-1} \}$ and recall
  the notation from \eqref{830}. Then, for every $a \in \mathbb R$,
  \begin{equation*}
    \begin{split}
      q_h(a) &\overset{\phantom{\eqref{275}}}{=}
      \mathbb P_a^\treegraph \big[|\mathcal C_\rot^h \cap
        \treegraphplus|<\infty,\varphi_\treegraph(\rot)<h\big]
      +\mathbb P_a^\treegraph \big[|\mathcal C_\rot^h \cap
        \treegraphplus|<\infty,\varphi_\treegraph(\rot)\ge h\big]
      \\& \overset{\phantom{\eqref{275}}}{=} \bbone_{(-\infty,h)}(a)  +
      \bbone_{[h,\infty)}(a) \,
      \mathbb P_a^\treegraph \big[ |  \mathcal C_\rot^h \cap U_{x_i}  |<\infty
        \text{ for } i=1,\ldots,d-1 \big]
      \\& \overset{\eqref{275}}{=} \bbone_{(-\infty,h)}(a)  +
      \bbone_{[h,\infty)}(a) \, \mathbb E^Y \Big[  \mathbb P_{\frac{a}{d-1}+
          Y}^\treegraph \big[ |\mathcal C_\rot^h \cap
          \treegraphplus|<\infty \big]  \Big]^{d-1} \stackrel[\eqref{250}]{\; \eqref{257} \; }{=} (R_h q_h)(a),
    \end{split}
  \end{equation*}
  completing the proof.
\end{proof}

Next, we give various necessary properties of  solutions to $R_h f = f$.
\begin{lemma}
  \label{l:supislessthanone}
  Let $h\in \mathbb R$. Assume that $f\in \mathcal S_h$ solves $R_h f = f$.
  Then $f$ is continuous and positive on $[h,\infty)$. If additionally $f$
  is not identically 1, then $\sup_{a\in [h,\infty)} f(a) < 1$ and
  $\lim_{a\to\infty} f(a) =0$.
\end{lemma}
\begin{proof}
  For the continuity and positivity we note that for $a \geq h$ one can write
  \begin{equation*}
    \begin{split}
      f(a) &= (R_h f)(a)
      \overset{\eqref{257}}{=}
      \bigg( \int_{\mathbb R} f
        \big(\tfrac{a}{d-1} +y \big) \tfrac{\sqrt{d-1}}{\sqrt{2\pi d}}
        e^{-\frac{(d-1)y^2}{2d}} \, dy \bigg)^{d-1} \\
      &= \bigg( \tfrac{\sqrt{d-1}}{\sqrt{2\pi d}}  e^{-\frac{a^2}{2d(d-1)} }
        \int_{\mathbb R} f(z) e^{-\frac{d-1}{2d}z^2 +  \frac{a}{d}z} \, dz
        \bigg)^{d-1}.
    \end{split}
  \end{equation*}
  The right hand side is continuous in $a$ by the dominated convergence
  theorem and it is also positive since $f$ is non-negative and equal to
  1 on $(-\infty,h)$.

  If $f$ is not identically 1, then there is some $b\geq h$ with $f(b)<1$.
  Hence, by the continuity of $f$ on $[h,\infty)$ previously shown, there
  is an interval of positive Lebesgue measure in $[h,\infty)$ on which $f$
  is strictly smaller than 1. Due to $f=R_h f$ and $0 \leq f \leq 1$,
  this implies that $f(a) < 1$ for all $a\geq h$ by the definition
  \eqref{257} of~$R_h$.

  We will now show that one even has $\sup_{[h,\infty)} f(a) < 1$.
  Consider the intervals $I_k \subseteq [h,\infty)$, $k\geq 0$, from
  \eqref{276}. Since $f<1$  and $f$ is continuous on $[h,\infty)$  and
  $I_0$ is compact, we have
  $\Delta \coloneqq \max\{\frac19  ,  \sup_{a\in I_0} f(a)\} < 1$. If we
  show by induction on $k\geq 0$ that
  $\sup_{a \in I_0\cup\ldots\cup I_k} f(a) \leq \Delta$ for all $k\geq 0$,
  then $\sup_{[h,\infty)} f(a) \leq \Delta < 1$ follows since
  $\bigcup_{k=0}^\infty I_k = [h,\infty)$. Now for $k=0$ the claim is
  true by definition of $\Delta$. So assume it holds for $k\geq 0$. Let
  $Y \sim \mathcal N(0,\frac{d}{d-1})$ and define
  $\varepsilon \coloneqq \mathbb P^Y[|Y|\ge 2]$. Observe that
  $\varepsilon< \tfrac 14$ because $d\geq 3$. For $a\in I_{k+1}$ we can
  estimate
  \begin{equation*}
    \begin{split}
      f(a)  &= (R_h f)(a) \overset{\eqref{257}}{=}  \mathbb E^Y \big[ f
        (\tfrac{a}{d-1} + Y ) \big]^{d-1}
      \leq \mathbb E^Y \big[ f (\tfrac{a}{d-1} + Y )  \big]^{2} \\
      &\leq \Big( \mathbb E^Y \big[ f (\underbrace{\tfrac{a}{d-1}+Y}_{\in I_k \text{ by
          }\eqref{e:Iks}})
          \bbone_{\{|Y| \leq 2\}}  \big]
        +     \mathbb E^Y \big[ \bbone_{\{|Y| > 2\}}  \big] \Big)^{2}
      \leq \big( \Delta \cdot (1-\varepsilon) + \varepsilon \big)^{2}
    \end{split}
  \end{equation*}
  by induction hypothesis. Therefore,
  \begin{equation*}
    \sup_{a\in I_{k+1}} f(a) -\Delta \leq
    \big( \Delta \cdot (1-\varepsilon) + \varepsilon   \big)^{2} - \Delta
    = \underbrace{(\Delta -1)}_{< \, 0}
    \underbrace{(\Delta \cdot (1-\varepsilon)^2 - \varepsilon^2)}_{\geq \, \frac19
      \cdot (\frac34)^2 - (\frac14)^2 \, \geq \, 0 }
    \leq 0.
  \end{equation*}
  This shows that $\sup_{a\in I_{k+1}} f(a) \leq \Delta$, which together
  with the induction hypothesis implies
  $\sup_{a \in I_0\cup\ldots\cup I_{k+1}} f(a) \leq \Delta$ and completes
  the proof of $\sup_{a\in [h,\infty)} f(a) < 1$.

  It remains to show $\lim_{a\to \infty} f(a) = 0$. The assumption
  $R_h f = f$ implies that
  \begin{equation}
    \label{4332}
    \limsup_{a\to \infty} f(a)
    = \limsup_{a\to \infty} (R_h f)(a) \overset{\eqref{257}}{=}
    \big(\limsup_{a\to \infty} \mathbb E^Y \big[ f (\tfrac{a}{d-1} + Y )
        \big] \big)^{d-1}.
  \end{equation}
  Since by Fatou's lemma (using $0\le f \le 1$)
  \begin{equation*}
    \limsup_{a\to \infty}\mathbb E^Y \big[ f (\tfrac{a}{d-1} + Y )   \big]
    \le  \mathbb E^Y \big[\limsup_{a\to \infty}  f (\tfrac{a}{d-1} + Y )   \big],
  \end{equation*}
  we have found
  \begin{equation*}
    \ell \coloneqq \limsup_{a\to \infty} f(a)
    \overset{\eqref{4332}}{\leq} \mathbb E^Y
    \big[\limsup_{a\to \infty}  f (\tfrac{a}{d-1} + Y )   \big] ^{d-1}
    = \mathbb E^Y[\ell]^{d-1} = \ell^{d-1}.
  \end{equation*}
  However, $\ell \in [0,1)$ since $\sup_{a \in [h, \infty)} f(a) < 1$.
  Therefore, the only possibility is $\ell =0$. Hence
  $\lim_{a\to \infty} f(a) = 0$ because $f$ is non-negative.
\end{proof}

The third step of the proof of Theorem~\ref{t:etacontinuity} is the
following statement of `transience'.
\begin{lemma}
  \label{l:transience}
  For $K\ge 1$ and $\Lambda \ge h$  let (see \eqref{300} for the notation)
  \begin{equation}
    \label{e:defAKLambda}
    A_k^{K,\Lambda} \coloneqq \big\{ 1 \leq |
      \mathcal Z_k^h | \leq K , \varphi_\treegraph(y) \leq \Lambda \text{ for
        all } y \in \mathcal Z_k^h    \big\} \quad \text{for } k\ge 0.
  \end{equation}
  Then for every $a \in \mathbb R$, $K\ge 1$ and $\Lambda \ge h$ one has
  \begin{equation*}
    \mathbb P_a^\treegraph [ \limsup_{k\to \infty}A_k^{K,\Lambda} ] = 0.
  \end{equation*}
\end{lemma}
\begin{proof}
  Observe that the events
  $B_k \coloneqq A_k^{K,\Lambda} \cap \bigcap_{n\geq k+1} (A_n^{K,\Lambda})^\mathsf{c}$
  for $k\geq 0$ are disjoint. Furthermore, denoting
  $S_\treegraph^+(\rot,k+1) \cap U_y \eqqcolon \{y_1,\ldots,y_{d-1} \}$
  for $y\in \mathcal Z_k^h$ and recalling the definition of $\mathcal F_k$
  from \eqref{e:defFk}, it holds that  for $a\in \mathbb R$ and $k\geq 0$
  \begin{equation*}
    \begin{split}
      \mathbb P_a^\treegraph [B_k]  &\overset{\phantom{\eqref{e:recstep}}}{\geq} \mathbb
      P_a^\treegraph \big[A_k^{K,\Lambda} \, , \, \mathcal Z_{k+1}^h = \emptyset
        \big]
      = \mathbb E_a^\treegraph\Big[ \bbone_{A_k^{K,\Lambda}} \, \mathbb
        P_a^\treegraph \big[ \mathcal Z_{k+1}^h = \emptyset   \, \big| \,  \mathcal F_k
          \big] \Big]  \\
      &\overset{\phantom{\eqref{e:recstep}}}{=}   \mathbb E_a^\treegraph\Big[
        \bbone_{A_k^{K,\Lambda}} \,  \mathbb P_a^\treegraph \big[  \bigcap_{y \in
            \mathcal Z_k^h} \bigcap_{i=1}^{d-1} \{ \varphi_\treegraph(y_i) < h\} \, \big|
          \,  \mathcal F_k \big] \Big]  \\
      &\overset{\eqref{e:recstep}}{=}  \mathbb E_a^\treegraph\Big[
        \bbone_{A_k^{K,\Lambda}} \, \prod_{y \in \mathcal Z_k^h} \mathbb P^Y \big[
          \tfrac{\varphi_\treegraph(y)}{d-1} + Y < h \big]^{d-1} \Big] \\
      &\overset{\phantom{\eqref{e:recstep}}}{\geq} \mathbb E_a^\treegraph\Big[
        \bbone_{A_k^{K,\Lambda}} \, \underbrace{\mathbb P^Y \big[
            \tfrac{\Lambda}{d-1} + Y < h \big]^{K(d-1)}}_{\eqqcolon c_{K,\Lambda}} \Big]  =
      c_{K,\Lambda} \, \mathbb P_a^\treegraph [A_k^{K,\Lambda}].
    \end{split}
  \end{equation*}
  Thus for $a\in \mathbb R$ we have
  \begin{equation*}
    \sum_{k=0}^\infty \mathbb P_a^\treegraph
    [A_k^{K,\Lambda}] \leq \frac{1}{c_{K,\Lambda}} \sum_{k=0}^\infty \mathbb
    P_a^\treegraph [B_k]  =  \frac{1}{c_{K,\Lambda}} \mathbb
    P_a^\treegraph \Big[
      \bigcup_{k\geq 0} B_k\Big] < \infty.
  \end{equation*}
  The claim then follows by  the Borel-Cantelli lemma.
\end{proof}

The next lemma proves Theorem \ref{t:uniqueness}. Before that, we
introduce for every $h\in \mathbb R$ the  functions
\begin{equation}
  \label{252}
  q_h^k(a) \coloneqq
  \mathbb P_a^\treegraph [ \mathcal Z_k^h = \emptyset]=
  \mathbb P_a^\treegraph \Big[ |\mathcal C_\rot^h \cap
    S_\treegraph^+(\rot,k)|=0 \Big] \quad \text{for $a \in \mathbb R$, $k\geq
    0$}.
\end{equation}
It can be easily seen that $q_h^k \in  \mathcal S_h$ for $k\geq 0$ and
$\bbone_{(-\infty,h)}(a) = q_h^0(a) \leq q_h^1(a) \leq q_h^2(a) \leq \ldots$
for $a \in \mathbb R$. In particular,
$\lim_{k\to \infty} q_h^k(a) = q_h(a)$ for all $a\in \mathbb R$ by
\eqref{250}. In addition, applying \eqref{258} to the function
$f=\bbone_{(-\infty,h)}$ implies that
\begin{equation}
  \label{831}
  q_h^k = R_h^k \bbone_{(-\infty,h)}
  \quad \text{for $k\geq 0$}.
\end{equation}

\begin{lemma}
  \label{l:uniquesolution}
  Let $h\in \mathbb R$. The only solutions in $\mathcal S_h$ to
  $R_h f = f$ are the function $q_h$ and the constant $1$ function. These
  two functions coincide if $h>h_\star$ and are distinct if $h<h_\star$.
\end{lemma}

\begin{proof}
  From Lemma~\ref{l:receqn} we know that $q_h \in \mathcal S_h$ and
  $R_h q_h = q_h$. The same is of course true for the constant 1
  function. We first claim that every solution in $\mathcal S_h$ to
  $R_h f =f$ satisfies $f\ge q_h$. Indeed, if $f \in \mathcal S_h$ is
  such a solution, then  $R_h^k f=f$ for all $k\geq 0$. Also, the fact
  that  $f \in \mathcal S_h$ implies $f \geq \bbone_{(-\infty,h)}$. Hence
  $f=R_h^k f \geq  R_h^k  \bbone_{(-\infty,h)} =q_h^k$ for all $k\geq 0$
  by \eqref{257} and \eqref{831}. By letting $k$ tend to infinity we find
  $f\geq q_h$, proving the claim. In particular, if $q_h \equiv 1$
  (e.g.~when $h> h_\star$, see below \eqref{280}), then we have
  $f\equiv 1$ and thus $R_h f=f $ has a unique solution.

  Now assume that $q_h \not \equiv 1$ (e.g.~when $h< h_\star$, see below
    \eqref{280}) and that $f\not\equiv 1$ is a solution to $R_h f = f$.
  We claim that $f=q_h$. As we have already shown $f\geq q_h$, it remains
  to prove $f\leq q_h$. To see this, observe that by
  Lemma~\ref{l:supislessthanone} we know that
  $\delta \coloneqq \sup_{a \in [h, \infty)} f(a) \in (0,1)$. Let
  $m\geq 0$ be such that
  $\delta \coloneqq \sup_{a \in [h, \infty)} f(a) \in \big[\frac{1}{2^{m+1}},\frac{1}{2^m}\big)$.
  Then for  $a\in \mathbb R$ and $k\geq 0$ one has
  \begin{align}
    f(a)
    & \overset{\phantom{\eqref{258}}}{=} (R_h^k f)(a)
    \overset{\eqref{258}}{=}  \mathbb E_a^\treegraph \Big[ \bbone_{\{ \mathcal
          Z_k^h = \emptyset \} } \prod_{y\in \mathcal Z_k^h} f(\varphi_\treegraph(y))
      \Big] +  \mathbb E_a^\treegraph \Big[ \bbone_{\{ \mathcal Z_k^h \neq
          \emptyset \} }  \prod_{y\in \mathcal Z_k^h} f(\varphi_\treegraph(y)) \Big]  \nonumber \\
    & \overset{\eqref{252}}{\leq} q_h^k(a) + \sum_{n\geq m} \frac{1}{2^n} \mathbb
    P_a^\treegraph \Big[\mathcal Z_k^h \neq \emptyset \, , \, \prod_{y\in \mathcal
        Z_k^h} f(\varphi_\treegraph(y)) \in
      \big[\tfrac{1}{2^{n+1}},\tfrac{1}{2^n}\big)   \Big].
    \label{260}
  \end{align}
  Note that for the events appearing on the right hand side of \eqref{260} one
  has
  \begin{equation}
    \label{261}
    \begin{split}
      &\big\{\mathcal Z_k^h \neq \emptyset \, , \, \textstyle \prod_{y\in \mathcal
          Z_k^h}  f(\varphi_\treegraph(y)) \in
        \big[\tfrac{1}{2^{n+1}},\tfrac{1}{2^n}\big)  \big\} \\
      &\quad \subseteq \big\{ | \mathcal Z_k^h | \geq 1 \, , \, \delta^{|\mathcal
          Z_k^h|} \geq \tfrac{1}{2^{n+1}} \, , \, f(\varphi_\treegraph(y)) \geq
        \tfrac{1}{2^{n+1}} \text{ for all } y \in \mathcal Z_k^h    \big\}    \\
      &\quad \subseteq \big\{ | \mathcal Z_k^h | \geq 1 \, , \, 2^{n+1} \geq
        (1/\delta)^{|\mathcal Z_k^h|} \, , \, f(\varphi_\treegraph(y)) \geq
        \tfrac{1}{2^{n+1}} \text{ for all } y \in \mathcal Z_k^h    \big\}    \\
      &\quad \subseteq \big\{ 1 \leq | \mathcal Z_k^h | \leq K_n \, , \,
        \varphi_\treegraph(y) \leq \Lambda_n \text{ for all } y \in \mathcal Z_k^h
        \big\}
      \overset{\eqref{e:defAKLambda}}= A_k^{K_n,\Lambda_n}
    \end{split}
  \end{equation}
  with $K_n \coloneqq  \log_{1/\delta} (2^{n+1})$ and
  $\Lambda_n \coloneqq \sup\{a\in \mathbb R \, | \, f(a) \geq \tfrac{1}{2^{n+1}} \}$.
  We observe that for $n\geq m$ it holds $K_n \geq 1$ since then
  $1/\delta \leq 2^{m+1} \leq 2^{n+1}$. Moreover,
  $h \leq \Lambda_n < \infty$ since $f\in \mathcal S_h$ (so $f=1$ on
    $(-\infty,h)$) and $\lim_{a\to \infty} f(a) = 0$  by
  Lemma~\ref{l:supislessthanone}. As a consequence,
  Lemma~\ref{l:transience} and Fatou's lemma imply
  $\lim_{k\to\infty}\mathbb P_a^\treegraph[A_k^{K_n,\Lambda_n}]=0$.
  Therefore, by using the dominated convergence theorem, for
  $a \in \mathbb R$ one has
  \begin{equation*}
    f(a)  \stackrel[\eqref{261}]{\eqref{260}}{\leq} \lim_{k\to \infty} \bigg(
      q_h^k(a) + \sum_{n\geq m} \frac{1}{2^n} \mathbb P_a^\treegraph \big[
        A_k^{K_n,\Lambda_n}  \big]   \bigg)  =
    q_h(a).
  \end{equation*}
  This implies that $f=q_h$, completing the proof.
\end{proof}

As last result of Section \ref{s:recursiveequation} we compute the
Fr\'echet derivative of the operators $(R_h)_{h\in \mathbb R}$ defined in
\eqref{257}. This technical result is one of the main ingredients for
proving the existence of exponential moments of $|\mathcal C_\rot^h|$ in
the subcritical phase (Section \ref{s:subcritical}). Incidentally, let us
mention that its proof is based on the hypercontractivity estimate
\eqref{405} and that the precise relation between $p$ and $q$ in the
estimate is vital (for $p=2$).

\begin{proposition}
  \label{1007}
  Let $h\in \mathbb R$ and consider the operator
  $R_h: L^2(\nu) \to L^2(\nu)$ from \eqref{257}. Then the Fr\'echet
  derivative of $R_h$ at $f\in L^2(\nu)$ is given by
  $A^f_h: L^2(\nu) \to L^2(\nu)$ with
  \begin{equation}
    \label{1005}
    A^f_h g \coloneqq \bbone_{[h,\infty)} \cdot (d-1) \mathbb E^Y
    [f(\tfrac{\cdot}{d-1} + Y )   ]^{d-2} \, \mathbb E^Y [g(\tfrac{\cdot}{d-1} + Y
        )    ].
  \end{equation}
  In particular, if $g\in L^2(\nu)$ vanishes on $(-\infty,h)$, then
  $A^1_h g = L_h g$, where $A^1_h$ is the Fr\'echet derivative of $R_h$
  at the constant function 1 and $L_h$ is given in \eqref{e:defLh}.
  Furthermore, for all $\varepsilon>0$ there exists $r>0$ such that
  $\|A^f_h g \|_{L^2(\nu)} \leq (\lambda_h  + \varepsilon) \, \|g\|_{L^2(\nu)}$
  if $g\in L^2(\nu)$ vanishes on $(-\infty,h)$ and
  $\|f-1\|_{L^2(\nu)} \leq r$.
\end{proposition}

\begin{proof}
  We start with some observations. For $u \in L^2(\nu)$ let us abbreviate
  $\hat u(a) \coloneqq \mathbb E^Y \big[ u (\tfrac{a}{d-1} + Y ) \big]$,
  $a\in \mathbb R$. We further set
  $p_i \coloneqq   \tfrac{(d-1)^2+1}{i} \geq 2$ for $i=1,\ldots,d-1$.
  Then for $u\in L^2(\nu)$ one has
  \begin{equation}
    \label{421}
    \| \hat u^{i} \|_{L^{p_i} (\nu)}   = \| \hat u \|^i_{L^{(d-1)^2+1} (\nu)}
    \overset{\eqref{405}}{\leq} \|  u \|^i_{L^{2} (\nu)} < \infty.
  \end{equation}
  Now if $u,v,w \in L^2(\nu)$ and $i,j,k \in \{0,\ldots,d-1\}$ with
  $i+j+k \leq d-1$, then one has $\hat u^i \in L^{p_i}(\nu)$,
  $\hat v^j \in L^{p_j}(\nu)$ and $\hat w^k \in L^{p_k}(\nu)$ by
  \eqref{421}, where we put $p_0\coloneqq \infty$, and therefore
  \begin{equation}
    \label{1004}
    \begin{split}
      \| \hat u^i \, \hat v^j \, \hat w^k \|_{L^2(\nu)}
      &\overset{2 \leq p_{i+j+k}}{\leq} \|  \hat u^i \, \hat v^j \, \hat w^{k}
      \|_{L^{p_{i+j+k}}(\nu)}
      \overset{(*)}{\leq} \| \hat u^i \|_{L^{p_i}(\nu)}  \| \hat v^j
      \|_{L^{p_j}(\nu)}     \| \hat w^{k} \|_{L^{p_{k}}(\nu)}   \\
      &\ \ \overset{\eqref{421}}{\leq} \ \  \| u \|^i_{{L^2}(\nu)} \| v
      \|^j_{{L^2}(\nu)}      \|  w  \|^{k}_{{L^2}(\nu)} < \infty,
    \end{split}
  \end{equation}
  where in $(*)$ we use the generalised H\"older inequality.

  To compute the Fr\'echet derivative of $R_h$ note that for
  $f,g\in L^2(\nu)$ one has
  \begin{equation}
    \label{1003}
    \begin{split}
      R_h(f+g) - R_h f
      &\overset{\eqref{257}}{=}  \bbone_{[h,\infty)}\cdot  \big( (\hat f + \hat
          g)^{d-1} -\hat f^{d-1} \big) \\
      &\overset{\phantom{\eqref{401}}}{=}  \bbone_{[h,\infty)}\cdot
      \sum_{i=0}^{d-2}\tbinom{d-1}{i}\hat f^i \hat g^{d-1-i}
      \overset{\phantom{\eqref{401}}}{=} A^f_h g + E^f_h g,
    \end{split}
  \end{equation}
  where $A^f_h g = \bbone_{[h,\infty)}\cdot (d-1) \hat f^{d-2} \hat g$ is
  the function defined in \eqref{1005} and the operator
  $E^f_h : L^2(\nu) \to L^2(\nu)$ is given by
  \begin{equation}
    \label{1002}
    E^f_h g \coloneqq  \bbone_{[h,\infty)}\cdot
    \sum_{i=0}^{d-3}\tbinom{d-1}{i}\hat f^i \hat g^{d-1-i}.
  \end{equation}
  Note that the map $A^f_h$ is linear and also bounded since
  $\sup_{\|g\|_{L^2(\nu)} \leq 1}  \|A^f_h g \|_{L^2(\nu)}
  \leq  {(d-1)} \sup_{\|g\|_{L^2(\nu)} \leq 1}
  \| \hat f^{d-2} \hat g  \|_{L^2(\nu)}< \infty$
  by \eqref{1004}. To conclude that $A^f_h$ is the Fr\'echet derivative
  of $R_h$ at $f$ it remains to show that
  \begin{equation}
    \label{425}
    \frac{\| R_h(f+g) - R_h f  - A^f_h g \|_{L^2(\nu)}}{\|g\|_{L^2(\nu)}}
    \overset{\eqref{1003}}{=}  \frac{\| E^f_h g \|_{L^2(\nu)}}{\|g\|_{L^2(\nu)}}
    \to   0 \quad \text{ if } \|g\|_{L^2(\nu)} \to 0.
  \end{equation}
  This is the case because
  \begin{equation*}
    \begin{split}
      \| E^f_h g \|_{L^2(\nu)}
      \stackrel[\eqref{1004}]{\eqref{1002}}{\leq} \sum_{i=0}^{d-3} \tbinom{d-1}{i} \|
      f \|^i_{{L^2}(\nu)} \|g \|^{d-1-i}_{{L^2}(\nu)},
    \end{split}
  \end{equation*}
  implying \eqref{425}. Thus $A^f_h$ is the Fr\'echet derivative of $R_h$ at $f$.

  From \eqref{1005} and \eqref{e:defLh} we directly see that
  $A^1_h g = L_h g$ if $g\in L^2(\nu)$ vanishes on $(-\infty,h)$. It
  remains to show the second part of the statement.  We have
  $\|A^f_h g \|_{L^2(\nu)} \leq \|A^f_h g - A^1_h g\|_{L^2(\nu)}  +  \|A^1_h g \|_{L^2(\nu)}$.
  For $g \in L^2(\nu)$ with  $g= 0$ on $(-\infty,h)$ one obtains
  $\|A^1_h g \|_{L^2(\nu)} = \|L_h g \|_{L^2(\nu)} \leq  \lambda_h \| g \|_{L^2(\nu)}$
  by \eqref{17}. Moreover, the formula
  $b^{d-2}-1 = (b-1)(1+b+\ldots+b^{d-3})$ and the triangle inequality imply
  \begin{equation*}
    \|A^f_h g - A^1_f g \|_{L^2(\nu)} \overset{\eqref{1005}}{\leq} (d-1)\, \| \hat
    g \, (\hat f^{d-2} -1) \|_{L^2(\nu)} \leq (d-1)\sum_{i=0}^{d-3} \| \hat g \,
    (\hat f - 1)\hat f^i \|_{L^2(\nu)}
  \end{equation*}
  and therefore $\|A^f_h g - A^1_f g \|_{L^2(\nu)} \leq (d-1) \| g \|_{L^2(\nu)}
  \|f-1\|_{L^2(\nu)} \sum_{i=0}^{d-3} \| f \|^i_{L^2(\nu)}$ by \eqref{1004}.
  All in all we showed
  \begin{equation}
    \label{1006}
    \|A^f_h g \|_{L^2(\nu)}  \leq  \Big( \lambda_h + (d-1)\|f-1\|_{L^2(\nu)}
      \sum_{i=0}^{d-3} \| f \|^i_{L^2(\nu)}\Big) \,  \|  g  \|_{L^2(\nu)}.
  \end{equation}
  Now let $\varepsilon>0$ and take $r>0$ such that
  $(d-1) ((1+r)^{d-2}-1) \leq \varepsilon$. Then if
  $\| f-1 \|_{L^2(\nu)} \leq r$, and hence also
  $\|f\|_{L^2(\nu)} \leq 1+r$, we have
  $\|A^f_h g \|_{L^2(\nu)} \leq (\lambda_h + \varepsilon) \, \|g\|_{L^2(\nu)}$
  by \eqref{1006}. This concludes the proof.
\end{proof}



\section{Behaviour of the level sets in the supercritical phase}
\label{s:supercritical}

In this section we study the behaviour of the level sets of the Gaussian
free field on $\treegraph$  for $h < h_\star$. The main goal is to show
that the percolation probabilities $\eta$ and $\eta^+$ are continuous
functions of the level $h$ on the interval $(-\infty, h_\star)$
(Theorem~\ref{t:etacontinuity}, corresponding to \eqref{e:defeta}) and to
prove that $| \mathcal C_{\rot}^{h}| $ grows exponentially in the radius
with probability bounded away from zero when $h<h_\star$
(Theorem~\ref{t:expgrowth}, corresponding to \eqref{e:preexpgrow}). Along
the way we also show the equivalence of the probabilities of forward
percolation and of a non-vanishing martingale limit (Proposition
  \ref{equivalence}, corresponding to \eqref{mainresultmartingale}).
These results  essentially come as an application of
Theorem~\ref{t:uniqueness} from Section \ref{s:recursiveequation}. For
this section recall the measure $\nu$ defined above \eqref{e:defLh}.



\subsection{Continuity of the percolation probability}
\label{ss:continuity}

In this section we analyse   the continuity properties of the percolation
probabilities $\eta$ and $\eta^+$, and show \eqref{e:defeta} in Theorem
\ref{t:etacontinuity}. Recall the functions $q_h$, $h \in \mathbb R$,
introduced in \eqref{250} and their relation with $\eta^+$ reported in
\eqref{280}.

\begin{theorem}
  \label{t:etacontinuity}
  The functions $\eta$ and $\eta^+$ are left-continuous on $\mathbb R$
  and continuous on $\mathbb R\setminus\{h_\star\}$.
\end{theorem}

\begin{proof}
  Note that
  \begin{equation}
    \label{A.1}
    \eta^+(h) = \mathbb P^\treegraph \Big[ \bigcap_{k\geq 1} \big\{
        \mathcal C_\rot^h\cap  S_\treegraph^+(\rot,k ) \neq \emptyset
        \big \}  \Big]
    = \lim_{k\to \infty} \mathbb P^\treegraph \big[
      \mathcal C_\rot^h\cap  S_\treegraph^+(\rot,k ) \neq \emptyset
      \big].
  \end{equation}
  Under $\mathbb P^\treegraph$ the vector
  $(\varphi_\treegraph(y))_{y\in B_\treegraph^+(\textup{o},k)}$ has a
  density  and thus
  $h\mapsto \mathbb P^\treegraph [\mathcal C_\rot^h\cap  S_\treegraph^+(\rot,k ) \neq \emptyset]$
  is a continuous function. Therefore by \eqref{A.1}, $\eta^+$ is a
  decreasing limit of continuous functions and hence upper
  semicontinuous. As $\eta^+$ is a non-increasing function, it is thus
  left-continuous. With the obvious changes in \eqref{A.1} one can also
  show the left-continuity of $\eta$.

  To show the right-continuity on $\mathbb R\setminus \{h_\star\}$
  observe first that  if $h>h_\star$, then $\eta(h) = \eta^+(h)=0$ by
  definition and the comment at the end of Section \ref{s:notation}. So
  it remains to prove the right-continuity on $(-\infty,h_\star)$. Fix
  $h<h_\star$ and assume $(h_\ell)_{\ell\geq0}$ is a sequence satisfying
  $h_\ell \downarrow h$ and $h_\ell < h_\star$ for all $\ell\geq 0$. We
  will show that $\lim_{\ell \to \infty} \eta^+(h_\ell) = \eta^+(h)$ and
  $\lim_{\ell \to \infty} \eta(h_\ell) = \eta(h)$. Observe that by
  \eqref{280} and the dominated convergence theorem the former follows
  from the claim
  \begin{equation}
    \label{1232}
    \lim_{\ell \to \infty} q_{h_\ell}(a)
    = q_h(a) \quad \text{for } a\in \mathbb R\setminus \{h\}.
  \end{equation}
  Actually, also $\lim_{\ell \to \infty} \eta(h_\ell) = \eta(h)$ follows
  from \eqref{1232} by a double application of the dominated convergence
  theorem since
  \begin{equation*}
    \begin{split}
      \eta(h_\ell)
      &\stackrel{\eqref{e:defPa}}{=}
      \int_{\mathbb R} \mathbb P_a^\treegraph [|\mathcal C_{\rot}^{h_\ell}|=\infty] \, d\nu(a)
      \stackrel{}{=}
      \int_{\mathbb R} \big(1-\mathbb P_a^\treegraph [|\mathcal C_{\rot}^{h_\ell}|<\infty] \big)
      \bbone_{[h_\ell,\infty)}(a) \, d\nu(a)  \\
      &\stackrel{\phantom{\eqref{e:defPa}}}{=}
      \int_{\mathbb R} \big(1-\mathbb P_a^\treegraph
        [|\mathcal C_{\rot}^{h_\ell} \cap U_{x_i}|<\infty \text{ for all } i=1,\ldots,d] \big)
      \bbone_{[h_\ell,\infty)}(a) \, d\nu(a)  \\
      &\stackrel[\eqref{250}]{\eqref{275}}{=}
      \int_{\mathbb R} \big(1-\mathbb E^Y[q_{h_\ell}(\tfrac{a}{d-1}+Y)]^d \big)
      \bbone_{[h_\ell,\infty)}(a) \, d\nu(a).
    \end{split}
  \end{equation*}
  Hence it remains to show \eqref{1232}.

  Define the two auxiliary functions $\widetilde q_h$ and $q'_h$ on
  $\mathbb R$ by
  \begin{equation}
    \label{283}
    \widetilde q_h(a) \coloneqq \lim_{\ell\to \infty} q_{h_\ell}(a)
    = \inf_{\ell\geq 0} q_{h_\ell}(a)   \quad \text{for $a \in \mathbb R$}
  \end{equation}
  and
  \begin{equation}
    \label{1233}
    q'_h(a) \coloneqq \begin{cases}
      \widetilde q_h(a), &\text{if } a \in \mathbb R\setminus \{h \} \\
      (R_ h\widetilde q_h)(h), &\text{if } a =h.
    \end{cases}
  \end{equation}
  We will now apply Theorem~\ref{t:uniqueness} to show $q'_h = q_h$. From
  this the claim \eqref{1232} follows by \eqref{1233} and \eqref{283}.

  Since $h_\ell < h_\star$, one has $q_{h_\ell} \not \equiv 1 $ for all
  $\ell \geq 0$ (see below \eqref{280}). This implies
  $\widetilde q_{h} \not \equiv 1 $ by \eqref{283} (being a decreasing
    limit) and hence also $q'_h \not \equiv 1$ by \eqref{1233}. Moreover
  if $a< h$, then $a<h_\ell$ for all $\ell \geq 0$, which yields
  $q_{h_\ell}(a)=1$ for all $\ell\geq0$. This implies $q'_{h}(a) = 1$
  for $a<h$ by \eqref{283} and \eqref{1233}. Thus $q'_h \in \mathcal S_h$.
  Finally, for $a>h$ and $\ell \geq 0$ such that $h_\ell \leq a$, one
  finds by Lemma \ref{l:receqn} and \eqref{257} that
  $q_{h_\ell}(a) = \mathbb E^Y \big[ q_{h_\ell} (\tfrac{a}{d-1} + Y )\big]^{d-1}$.
  If we let $\ell$ tend to infinity on both sides, then \eqref{283} and
  the dominated convergence theorem give
  $\widetilde q_{h}(a) =
  \mathbb E^Y \big[ \widetilde q_{h} (\tfrac{a}{d-1} + Y )  \big]^{d-1}$
  for all $a>h$. This together with \eqref{1233} shows $q'_h= R_h q'_h$.
  By Theorem~\ref{t:uniqueness} we conclude that $q'_h = q_h$. The proof
  is complete.
\end{proof}



\subsection{Percolation probability and non-triviality of the martingale limit}
\label{ss:martingalelimit}

Recall the martingale $(M^{\geq h}_k)_{k\geq 0}$ from \eqref{209}. We now
apply Theorem \ref{t:uniqueness} from Section~\ref{s:recursiveequation}
to show  in Proposition \ref{equivalence} the equivalence
\eqref{mainresultmartingale} between the probability of non-vanishing of
the martingale limit and $\eta^+(h)$. From the discussion at the end of
Section \ref{subsection1.1} we already know that
$\mathbb P^\treegraph \big[ M^{\geq h}_\infty > 0\big] = \eta^+(h) = 0$
for $h>h_\star$. We now prove that the first equality remains true also
if $h<h_\star$.

\begin{proposition}
  \label{equivalence}
  One has
  \begin{equation}
    \label{222}
    \eta^+(h) = \mathbb P^\treegraph \big[ M^{\geq h}_\infty > 0\big] \quad
    \text{for all $h\in \mathbb R \setminus \{h_\star\}$}.
  \end{equation}
\end{proposition}
\begin{proof}
  For every $h\in \mathbb R$ we introduce the  function
  $m_h(a) \coloneqq   \mathbb P_a^\treegraph \big[ M^{\geq h}_\infty = 0 \big]$
  for $a \in \mathbb R$, where  $\mathbb P_a^\treegraph$ is the
  conditional probability defined in \eqref{e:defPa}. We  note that
  \begin{equation}
    \label{279}
    \int_{\mathbb R} m_h(a) \, d\nu(a) =
    \mathbb E^\treegraph[m_h(\varphi_\treegraph(\rot))]
    \overset{\eqref{e:defPa}}{=}  \mathbb P^\treegraph \big[ M^{\geq h}_\infty =
      0\big].
  \end{equation}
  By \eqref{280} and \eqref{279} it is enough to show that for
  $h\neq h_\star$ one has $q_h = m_h$. This will follow from  Theorem
  \ref{t:uniqueness}. Note that $m_h \in \mathcal S_h$ since
  $m_h(a) \geq \mathbb P_a^\treegraph[\mathcal Z_0^h = \emptyset]
  = \bbone_{(-\infty,h)}(a)$
  by \eqref{209} and \eqref{e:defPa}. We also have that $R_h m_h = m_h$.
  Indeed, recall \eqref{830} and denote
  $S_\treegraph^+(\rot,1) \eqqcolon \{x_1,\ldots,x_{d-1} \}$. Let us
  write
  $M_{k,i}^{\geq h} \coloneqq \lambda_h^{-k}
  \sum_{y\in  \mathcal Z_k^h \cap U_{x_i}} \chi_h(\varphi_\treegraph(y)) $
  for $k\geq 1$ and $i=1,\ldots,d-1$, so that
  $M_k^{\geq h} = \sum_{i=1}^{d-1} M_{k,i}^{\geq h}$ for $k\geq 1$. Then
  for $a \in \mathbb R$
  \begin{equation*}
    \begin{split}
      m_h(a) &\overset{\phantom{\eqref{e:defPa}}}{=}   \mathbb P_a^\treegraph
      \big[\varphi_\treegraph(\rot)<h \, , \, M_\infty^{\geq h} = 0 \big] +
      \mathbb P_a^\treegraph \big[\varphi_\treegraph(\rot)\geq h \, , \,
        M_\infty^{\geq h} = 0 \big] \\
      &\stackrel{\eqref{e:defPa}}{=} \bbone_{(-\infty,h)}(a) +
      \bbone_{[h,\infty)}(a) \, \mathbb P_a^\treegraph \big[ \lim_{k\to \infty}
        M_{k,i}^{\geq h} = 0 \text{ for } i=1,\ldots,d-1 \big] \\
      &\overset{\eqref{275}}{=} \bbone_{(-\infty,h)}(a)  +
      \bbone_{[h,\infty)}(a) \, \mathbb E^Y \Big[  \mathbb P_{\frac{a}{d-1}+
          Y}^\treegraph \big[ M_\infty^{\geq h} = 0 \big]  \Big]^{d-1}
      \overset{\, \eqref{257} \,}{=} (R_h m_h)(a).
    \end{split}
  \end{equation*}
  Now if $h>h_\star$, then by Theorem~\ref{t:uniqueness} we find
  $m_h = q_h \equiv 1$. On the other hand, if $h<h_\star$, then
  \eqref{279} and \eqref{278} imply that $m_h$ is not the constant 1
  function and so $m_h=q_h$ by Theorem~\ref{t:uniqueness} again. The
  proof is complete.
\end{proof}



\subsection{Geometrical growth
  \texorpdfstring{of $| \mathcal C_{\normalfont \rot}^{h}|$}{}
  in the supercritical phase}
\label{ss:geometricgrowth}

We come to the proof of \eqref{e:preexpgrow}, essentially that   for
$h<h_\star$ the number of vertices in $\treegraphplus$ connected over
distance $k$ above level $h$ to the root $\rot \in \treegraph$ grows
exponentially  in $k$ with positive probability. Recall the notation from
\eqref{300}.

\begin{theorem}
  \label{t:expgrowth}
  Let $h < h_\star$ (so that $\lambda_h>1$, see
    Proposition~\ref{p:Lhprops}). Then
  \begin{equation*}
    \lim_{k \to \infty}
    \mathbb P^\treegraph \Big[ \big| \mathcal Z_k^h \big| \geq
      \frac{\lambda_h^{k}}{k^2} \Big]  =  \eta^+(h) > 0.
  \end{equation*}
\end{theorem}

\begin{proof}
  Note that one directly has
  \begin{equation*}
    \limsup_{k \to \infty} \,  \mathbb P^\treegraph \Big[ \big|  \mathcal Z_k^h
      \big| \geq \frac{\lambda_h^{k}}{k^2} \Big]  \leq    \limsup_{k \to \infty} \,
    \mathbb P^\treegraph \big[ \mathcal C_{\rot}^{h} \cap
      S_\treegraph^+(\rot,k)  \neq \emptyset  \big] \overset{\eqref{251}}{=} \eta^+(h).
  \end{equation*}
  Thus we only have  to find a corresponding lower bound. By Fatou's lemma
  \begin{equation}
    \label{211}
    \begin{split}
      \eta^+(h)
      &\overset{\eqref{222}}{=} \mathbb P^\treegraph \big[ M^{\geq h}_\infty > 0\big]
      \leq \mathbb P^\treegraph \big[ M^{\geq h}_k \geq \tfrac1k
        \text{ for all } k \text{ large enough}\big]
      \\ &  \overset{\phantom{\eqref{222}}}{\leq}
      \liminf_{k\to \infty} \mathbb P^\treegraph \big[ M^{\geq h}_k
        \geq \tfrac1k \big]
      \\&\overset{\phantom{\eqref{222}}}{\leq} \liminf_{k \to \infty} \Big( \mathbb P^\treegraph
        \big[M^{\geq h}_k \geq \tfrac1k \, , \,   A_k^h \big] + \mathbb P^\treegraph
        \big[  (A_k^h)^\mathsf{c}  \big] \Big),
    \end{split}
  \end{equation}
  where we introduced the event
  \begin{equation*}
    A_k^h \coloneqq \Big\{
      \sup_{ y \in S_\treegraph^+(\rot,k)}
      \chi_h(\varphi_\treegraph(y)) \leq k \Big\} \quad\text{for } k\geq 0.
  \end{equation*}
  On the event $A_k^h$ the inequality $M^{\geq h}_k \geq \frac1k$ implies
  $| \mathcal Z_k^h | \geq \frac{\lambda_h^{k}}{k^2}$ by \eqref{209}. Hence
  \begin{equation}
    \label{212}
    \mathbb P^\treegraph \big[M^{\geq h}_k \geq \tfrac1k \, , \,   A_k^h \big]
    \leq \,  \mathbb P^\treegraph \Big[ \big| \mathcal Z_k^h \big| \geq
      \frac{\lambda_h^{k}}{k^2} \Big] .
  \end{equation}
  To deal with the event $(A_k^h)^\mathsf{c}$ note that by
  Proposition~\ref{p:chibounds} and Remark~\ref{r:slopeofchi} (here
    $h<h_\star$) one has $\chi_h(a) \leq c_h a$ for $a\geq h$ and
  $\chi_h(a) =0$ for $a<h$. Thus, for  $y \in \treegraph$ and for $k\geq 0$
  \begin{equation*}
    \begin{split}
      &\mathbb P^\treegraph \big[   \chi_h(\varphi_\treegraph(y)) > k      \big]
      = \mathbb P^\treegraph \big[   \chi_h(\varphi_\treegraph(y)) > k \, , \,
        \varphi_\treegraph(y) \geq h   \big] \\
      &\qquad \leq \mathbb P^\treegraph \big[  c_h \varphi_\treegraph(y) > k
        \, , \,  \varphi_\treegraph(y) \geq h   \big]
      \overset{\eqref{0.2}}{\leq}  \exp\Big(-\frac{
          k^2}{2c_h^2g_\treegraph(\rot,\rot)} \Big),
    \end{split}
  \end{equation*}
  where in the last step we use the exponential Markov inequality. Hence,
  by a union bound, for $k\geq 0$
  \begin{equation}
    \label{210}
    \mathbb P^\treegraph \big[  (A_k^h)^\mathsf{c} \big]  \leq   \underbrace{\big|
      S_\treegraph^+(\rot,k) \big|}_{= (d-1)^k} \exp\Big(-\frac{
        k^2}{2c_h^2g_\treegraph(\rot,\rot)} \Big) \xrightarrow{k \to
      \infty} 0.
  \end{equation}
  From \eqref{211}, \eqref{212} and \eqref{210} we have that
  $\liminf_{k \to \infty}
  \mathbb P^\treegraph \big[ |  \mathcal Z_k^h |
    \geq \frac{\lambda_h^{k}}{k^2} \big] \geq \eta^+(h)$
  and the proof of Theorem~\ref{t:expgrowth} follows.
\end{proof}



\section{Exponential moments
  \texorpdfstring{of $|\mathcal C_{\normalfont \rot}^{h}|$}{}
  in the subcritical phase}
\label{s:subcritical}

This section proves that for every $h>h_\star$ the cardinality of the
connected component  of the level set of $\varphi_\treegraph$ above level
$h$ in $\treegraphplus$ containing the root $\rot\in \treegraph$ has
exponential moments and  actually, as a function of the value of
$\varphi_\treegraph(\rot)$, these exponential moments do not grow too
fast. This is the content of Theorem \ref{t:exponentialboundong} below
(corresponding to \eqref{e:mainresultsubcritical}). In its proof we will
use Proposition \ref{1007} from Section \ref{s:recursiveequation}.

To state the result, we define for every $h \in \mathbb R$ and $\delta> 0$
the (potentially infinite) function
\begin{equation}
  \label{573}
  g_{h,\delta}(a) \coloneqq  \mathbb E_a^\treegraph \Big[ (1+\delta)^{| \mathcal
      C_{\rot}^{h}\cap \treegraphplus   |} \Big]  \quad \text{for
    $a\in \mathbb R$,}
\end{equation}
where we use the notation for the conditional distribution of
$\varphi_\treegraph$ given $\varphi_\treegraph(\rot)=a$ defined in
\eqref{e:defPa}.  Observe that (recall $\nu$ from above \eqref{e:defLh})
\begin{equation}
  \label{832}
  \int_{\mathbb R} g_{h,\delta}(a) \, d\nu(a)
  = \mathbb
  E^\treegraph[g_{h,\delta}(\varphi_\treegraph(\rot))]
  \overset{\eqref{e:defPa}}{=} \mathbb E^\treegraph \Big[ (1+\delta)^{| \mathcal
      C_{\rot}^{\treegraph,h}\cap \treegraphplus   |} \Big].
\end{equation}

Note that if $q_h(a)<1$ for $q_h$ from \eqref{250}  (in particular this
  is  the case in the supercritical phase $h<h_\star$ for $a\ge h$), then
$g_{h,\delta }(a)$ is infinite. The main goal of this section is to show
that in the subcritical phase $h>h_\star$ there exists $\delta>0$ such
that the right hand side of \eqref{832} is finite and such that
$g_{h,\delta}(a)$ does not grow too fast as $a$ tends to infinity. Recall
the space $L^2(\nu)$ defined in \eqref{e:defLh}.

\begin{theorem}
  \label{t:exponentialboundong}
  Let $h > h_\star$. Then there exists $\delta_h>0$ such that
  \begin{equation}
    \label{e:gisinl2}
    g_{h,\delta_h} \in L^2(\nu).
  \end{equation}
  Moreover, $g_{h,\delta_h}$ equals 1 on $(-\infty,h)$ and
  $g_{h,\delta_h}(a)$ is finite for all $a\in \mathbb R$. Finally,
  $g_{h,\delta_h}$ is continuous on $[h,\infty)$ and for all $\gamma > 0$
  there exist $c_{h,\gamma}>0$ and $c_{h,\gamma}'>0$ such that
  \begin{equation}
    \label{415}
    g_{h,\delta_h}(a) \leq c_{h,\gamma} \exp(c_{h,\gamma}' a^{1+\gamma}) \quad
    \text{for all $a\geq h$}.
  \end{equation}
  In particular, \eqref{832} and \eqref{e:gisinl2}  imply
  $\mathbb E^\treegraph
  \big[ (1+\delta_h)^{| \mathcal C_{\rot}^{h}\cap \treegraphplus  |} \big]
  <\infty$.
\end{theorem}

\begin{remark}
  \label{r:linktointroduction}
  Note that \eqref{e:preexptail} follows from Theorem
  \ref{t:exponentialboundong} by the exponential Markov inequality. More
  precisely, for $h<h_\star$ and $a\in \mathbb R$ take say $\gamma=1$ in
  \eqref{415}. Then
  $\mathbb P^\treegraph\big[|\mathcal C_{\rot}^{h}|\ge k
    \,\big |\, \varphi_\treegraph(\rot) = a\big]
  \leq \mathbb P_a^\treegraph\big[| \mathcal C_{\rot}^{h}
    \cap \treegraphplus   |\ge k \big]
  \leq (1+\delta_h)^{-k} c_{h} \exp(c_h' a^2)$,
  thus \eqref{e:preexptail}. \qed
\end{remark}

The proof of Theorem~\ref{t:exponentialboundong} is split into various
lemmas. The first one characterises $g_{h,\delta}$ as a monotone limit of
functions in $L^2(\nu)$ which are obtained via iterated applications of a
certain operator $R_{h,\delta}$ (see \eqref{401}) to the constant 1
function (Lemma \ref{l:iteration}). The second lemma shows that for
$h>h_\star$ we can choose $\delta>0$ such that the operator $R_{h,\delta}$
is a strict contraction on a closed subset of $L^2(\nu)$ including the
constant 1 function (Lemma \ref{contractionlemma}). This is an
application of the technical Proposition \ref{1007} from Section
\ref{s:recursiveequation}. The combination of these two results will
quickly lead  to \eqref{e:gisinl2} via the Banach-Caccioppoli fixed-point
theorem and to the other properties of $g_{h,\delta}$ stated in Theorem
\ref{t:exponentialboundong} except for \eqref{415}. This is the content
of Corollary~\ref{5432}. It then remains to prove \eqref{415}. We first
show a weaker statement  in which $\gamma =1$ on the right hand side
(Lemma \ref{5555}). It  implies a recursive bound on $g_{h,\delta}$
(Lemma \ref{5678}) which subsequently can be used to show the stronger
statement  (Lemma \ref{9876}).

Let us  introduce for every $h \in \mathbb R$ and
$\delta> 0$ the  operator $R_{h,\delta}$ on $L^2(\nu)$ through:
\begin{equation}
  \label{401}
  \begin{split}
    &(R_{h,\delta} f) (a) \coloneqq \bbone_{(-\infty,h)}(a)  +
    \bbone_{[h,\infty)}(a) \cdot (1+\delta)  \mathbb E^Y \big[ f
      (\tfrac{a}{d-1} + Y ) \big]^{d-1}  \\
    &\text{for } f \in L^2(\nu)  \text{ and } a \in \mathbb R,
  \end{split}
\end{equation}
where, as usual, $Y \sim \mathcal N(0,\tfrac{d}{d-1})$. By the same
observations as below \eqref{257} one can check that indeed
$R_{h,\delta} f \in L^2(\nu)$ for $f\in L^2(\nu)$. Note also that
$R_{h,\delta}$ for $\delta> 0$  can be expressed in terms of the operator
$R_h$ from \eqref{257} via
$R_{h,\delta} = (\bbone_{(-\infty,h)} + \bbone_{[h,\infty)} \cdot (1+\delta) ) R_h$.
The role of $R_{h,\delta }$ can be seen from the following lemma.

\begin{lemma}
  \label{l:iteration}
  Let $h\in \mathbb R$, $\delta>0$ and define the (bounded) functions
  \begin{equation}
    \label{400}
    g_{h,\delta}^k(a) \coloneqq  \mathbb E_a^\treegraph \Big[ (1+\delta)^{|
        \mathcal C_{\rot}^{h}\cap B_\treegraph^+(\rot,k)   |}
      \Big]  \quad \text{for $a \in \mathbb R$, $k\geq 0$}.
  \end{equation}
  Then one has
  \begin{equation}
    \label{e:obvious}
    1\leq g^0_{h,\delta }\le g^{1}_{h,\delta }
    \leq  g^{2}_{h,\delta } \leq \ldots \leq g_{h,\delta}
    \qquad \text{and} \qquad \lim_{k\to \infty} g^k_{h,\delta} = g_{h, \delta}.
  \end{equation}
  Moreover, for every $k\ge 0$ one has
  \begin{equation}
    \label{e:kprogenyrr}
    g_{h,\delta }^k = R_{h,\delta}^{k+1} 1
    \qquad \text{and}\qquad   g_{h,\delta}^{k+1} = R_{h,\delta} g_{h,\delta}^k.
  \end{equation}
\end{lemma}
\begin{proof}
  The first part of  \eqref{e:obvious} is clear by definition and the
  second part follows by the monotone convergence theorem. Claim
  \eqref{e:kprogenyrr} can be seen via induction on $k\geq 0$. Indeed,
  for $k=0$ it holds
  $R_{h,\delta} 1 = \bbone_{(-\infty,h)}
  + \bbone_{[h,\infty)}\cdot (1+\delta) = g^0_{h,\delta}$
  by \eqref{401} and \eqref{e:defPa}. Furthermore, for $k\geq 0$ and
  $a\in \mathbb R$ one has (recall \eqref{830} and denote
    $S_\treegraph^+(\rot,1) \eqqcolon \{x_1,\ldots,x_{d-1} \}$)
  \begin{equation*}
    \begin{split}
      &g_{h,\delta}^{k+1}(a)
      =  \mathbb E_a^\treegraph \Big[ \big(
          \bbone_{\{\varphi_\treegraph(\rot)<h\}} +
          \bbone_{\{\varphi_\treegraph(\rot)\geq h\}} \big) \, (1+\delta)^{|
          \mathcal C_{\rot}^{h}\cap B_\treegraph^+(\rot,k+1)   |}
        \Big]  \\
      &\overset{\eqref{e:defPa}}{=} \bbone_{(-\infty,h)}(a)  +
      \bbone_{[h,\infty)}(a) \cdot (1+\delta)  \mathbb E_a^\treegraph \Big[
        \prod_{i=1}^{d-1} (1+\delta)^{| \mathcal C_{\rot}^{h}\cap
          B_\treegraph^+(\rot,k+1) \cap U_{x_i}   |}\Big] \\
      &\overset{\eqref{275}}{=} \bbone_{(-\infty,h)}(a)  +
      \bbone_{[h,\infty)}(a) \cdot (1+\delta) \mathbb E^Y \bigg[ \mathbb
        E_{\frac{a}{d-1}+Y}^\treegraph \Big[  (1+\delta)^{| \mathcal
            C_{\rot}^{\treegraph,h}\cap B_\treegraph^+(\rot,k)   |} \Big]
        \bigg]^{d-1} \\
      &\stackrel[\eqref{401}]{\; \eqref{400} \;}{=} (R_{h,\delta} \, g_{h,\delta}^k)(a)
      \overset{(*)}{=} (R_{h,\delta}^{k+2} 1 )(a),
    \end{split}
  \end{equation*}
  where in $(*)$ we use the induction hypothesis. This shows the first
  half of \eqref{e:kprogenyrr}, which implies the second half.
\end{proof}

For the next lemma we define for $h\in \mathbb R$ and $r>0$
\begin{equation*}
  B_{h,r} \coloneqq \{f\in L^2(\nu) \, | \, f\geq 1, \, f \textup{ equals 1 on }
    (-\infty,h) \textup{ and } \| f - 1 \|_{L^2(\nu)} \leq r \}.
\end{equation*}
Since $B_{h,r}$ is a closed subset of $L^2(\nu )$, it is a complete
metric space.

\begin{lemma}
  \label{contractionlemma}
  Let $h>h_\star$. Then there exists $\delta_h > 0$ and $r_h>0$ such that
  $R_{h,\delta_h}$ is a (strict) contraction on the complete metric space
  $B_{h,r_h}$. In particular, by the Banach-Caccioppoli fixed-point
  theorem there exists a unique $f^\star \in B_{h,r_h}$ with
  $R_{h,\delta_h} f^\star = f^\star$ and for all $f\in B_{h,r_h}$ one has
  $\| R_{h,\delta_h}^k f - f^\star \|_{L^2(\nu)} \to 0$ as $k\to \infty$.
\end{lemma}

\begin{proof}
  Let $\delta>0$ and consider $f,g \in L^2(\nu)$. By the relationship
  between $R_{h,\delta}$ and $R_h$ explained below \eqref{401} one has
  \begin{equation}
    \label{1008}
    \begin{split}
      &\| R_{h,\delta}g  - R_{h,\delta} f \|_{L^2(\nu)}
      =  \| (\bbone_{(-\infty,h)} + \bbone_{[h,\infty)} \cdot (1+\delta) )
      (R_h g - R_h f) \|_{L^2(\nu)} \\
      &\qquad \overset{\phantom{\eqref{1003}}}{\leq}  (1+\delta)  \cdot  \| R_h g  -
      R_h f \|_{L^2(\nu)}
      = (1+\delta) \cdot  \| R_h (f + g - f) - R_h f \|_{L^2(\nu)} \\
      &\qquad \overset{\eqref{1003}}{\leq}   (1+\delta)  \Big( \| A^f_h (g-f)
        \|_{L^2(\nu)} +   \| E^f_h (g-f) \|_{L^2(\nu)}    \Big) .
    \end{split}
  \end{equation}
  Since $h>h_\star$ (and thus $\lambda_h<1$ by
    Proposition~\ref{p:Lhprops}), we can choose $\varepsilon_h>0$ such
  that $\lambda_h + 2 \varepsilon_h < 1$. Now on the one hand, by
  Proposition \ref{1007} there exists $s_h>0$ such that
  $\| A^f_h (g-f) \|_{L^2(\nu)} \leq (\lambda_h + \varepsilon_h) \, \| g-f \|_{L^2(\nu)}$
  for $f,g \in B_{h,s_h}$, because then $f-g$ vanishes on $(-\infty,h)$
  and $\|f-1 \|_{L^2(\nu)} \leq s_h$. On the other hand, by \eqref{425}
  there exists $s'_h>0$ such that
  $\| E^f_h (g-f) \|_{L^2(\nu)} \leq \varepsilon_h  \| g-f \|_{L^2(\nu)}$
  if $\| g-f \|_{L^2(\nu)} \leq s'_h$. Hence if $f,g \in B_{h,r_h}$ with
  $r_h \coloneqq \frac12 \min\{s_h, s'_h \}$, then both conditions are
  simultaneously satisfied and one has
  \begin{equation}
    \label{1009}
    \begin{split}
      \| R_{h,\delta}g  &- R_{h,\delta} f \|_{L^2(\nu)}
      \overset{\eqref{1008}}{\leq}   (1+\delta) ( \lambda_h + 2\varepsilon_h ) \, \|
      g - f \|_{L^2(\nu)}.
    \end{split}
  \end{equation}
  Moreover, since $R_{h,\delta} 1 = 1 + \bbone_{[h,\infty)} \delta$ by
  \eqref{401}, one also has for $f\in B_{h,r_h}$ that
  \begin{equation}
    \label{1010}
    \begin{split}
      \| R_{h,\delta}f  - 1 \|_{L^2(\nu)}
      &\overset{\phantom{\eqref{1009}}}{\leq} \| R_{h,\delta}f  - R_{h,\delta} 1
      \|_{L^2(\nu)}  + \| R_{h,\delta} 1  - 1 \|_{L^2(\nu)}   \\
      &\stackrel{\eqref{1009}}{\leq}  (1+\delta) ( \lambda_h + 2\varepsilon_h ) r_h
      +   \delta.
    \end{split}
  \end{equation}
  Due to $\lambda_h + 2\varepsilon_h < 1$, we can choose
  $\delta = \delta_h >0$ such that
  $(1+\delta_h) ( \lambda_h + 2\varepsilon_h ) r_h  +  \delta_h \leq r_h$.
  This also implies
  $(1+\delta_h)(\lambda_h+2\varepsilon_h) \eqqcolon \Delta_h < 1$. Then
  $R_{h,\delta_h}$ maps the space $B_{h,r_h}$ to itself. Indeed, for
  $f\in B_{h,r_h}$ one has $R_{h,\delta_h}f \geq 1$ and
  $R_{h,\delta_h}f = 1$ on $(-\infty,h)$ by definition of $R_{h,\delta_h}$,
  and furthermore $\| R_{h,\delta_h}f  - 1 \|_{L^2(\nu)} \leq r_h$ by
  \eqref{1010}. Finally, \eqref{1009} shows
  $\| R_{h,\delta_h}g  - R_{h,\delta_h} f \|_{L^2(\nu)}
  \leq \Delta_h  \| g - f \|_{L^2(\nu)}$
  for $f,g \in B_{h,r_h}$, i.e. that $R_{h,\delta_h}$ is a strict
  contraction.
\end{proof}

With Lemma \ref{l:iteration} and Lemma \ref{contractionlemma} at hand, we
can readily show the first half of Theorem \ref{t:exponentialboundong}.

\begin{corollary}
  \label{5432}
  Let $h > h_\star$. Then there exists $\delta_h>0$ such that
  $g_{h,\delta_h} \in L^2(\nu)$. Moreover, $g_{h,\delta_h}$ equals 1 on
  $(-\infty,h)$, satisfies
  $R_{h,\delta_h} g_{h,\delta_h} = g_{h,\delta_h}$ and is continuous.
  Finally,  $g_{h,\delta_h}(a)$ is finite for all $a\in \mathbb R$.
\end{corollary}

\begin{proof}
  Consider $\delta_h>0$, $r_h >0$ and $f^\star \in L^2(\nu)$ from
  Lemma~\ref{contractionlemma}. We start by showing that $\nu$-almost
  everywhere  $f^\star = g_{h,\delta_h}$ and hence
  $g_{h,\delta_h} \in L^2(\nu)$. Note that by Lemma
  \ref{contractionlemma} one has $1\leq f^\star$ and thus by
  \eqref{e:kprogenyrr} also
  $g^{k}_{h,\delta_h} = R_{h,\delta_h}^{k+1} 1 \leq R_{h,\delta_h}^{k+1}  f^\star = f^\star$
  for all $k\geq 0$. By \eqref{e:obvious} and the monotone convergence
  theorem this shows
  \begin{equation}
    \label{419}
    \| \lim_{k\to \infty} g^{k}_{h,\delta_h} - f^\star \|_{L^2(\nu)}
    = \lim_{k\to \infty} \| g^{k}_{h,\delta_h} - f^\star \|_{L^2(\nu)}
    \overset{\eqref{e:kprogenyrr}}{=}
    \lim_{k\to \infty} \| R^k_{h,\delta_h}1 - f^\star \|_{L^2(\nu)}.
  \end{equation}
  Now since $1\in B_{h,r_h}$, we know from Lemma \ref{contractionlemma}
  that
  $\lim_{k\to \infty} \| R^k_{h,\delta_h}1 - f^\star \|_{L^2(\nu)} = 0$.
  Hence
  $\| \lim_{k\to \infty} g^{k}_{h,\delta_h} - f^\star \|_{L^2(\nu)} = 0$
  by \eqref{419} and so $\nu$-almost everywhere
  $f^\star = g_{h,\delta_h}$ by \eqref{e:obvious}. By \eqref{573} and
  \eqref{e:defPa}  it is obvious that $g_{h,\delta_h}$ equals 1 on
  $(-\infty,h)$. Now if we take $k$ to infinity on both sides of the
  equation $g_{h,\delta_h }^{k+1} = R_{\delta_h ,h} g_{h,\delta_h }^k$
  from \eqref{e:kprogenyrr}, we obtain
  \begin{equation}
    \label{e:gisfixedpoint}
    g_{h,\delta_h } = R_{h,\delta_h }g_{h,\delta_h }
  \end{equation}
  by \eqref{e:obvious} and the monotone convergence theorem. The right
  hand side of \eqref{e:gisfixedpoint} satisfies
  $(R_{h,\delta_h}g_{h,\delta_h})(a) < \infty $  for all $a\in \mathbb R$
  by \eqref{401} and \eqref{e:gisinl2}. Hence $g_{h,\delta_h}(a) < \infty$
  for all $a \in \mathbb R$. With \eqref{e:gisfixedpoint} established, we
  can show the continuity of $g_{h,\delta_h}$ on $[h,\infty)$ in the same
  way as the continuity of $f$ in Lemma~\ref{l:supislessthanone}.  Hence
  the proof of Corollary \ref{5432} is complete.
\end{proof}

It remains to prove \eqref{415}. In the next lemma we show a weaker
statement by applying results from Corollary \ref{5432}.

\begin{lemma}
  \label{5555}
  Let $h>h_\star$ and consider the function $g_{h,\delta_h} \in L^2(\nu)$
  from Corollary \ref{5432}. For all $\zeta>0$ there exists
  $c_{h,\zeta}>0$ such that
  \begin{equation}
    \label{409}
    g_{h,\delta_h}(a) \leq c_{h,\zeta} \exp(\zeta a^2) \quad \text{for $a \geq h$}.
  \end{equation}
\end{lemma}

\begin{proof}
  We will first show that $g_{h,\delta_h} \in L^q(\nu)$ for all $q \geq 1$,
  which will then imply \eqref{409}. Let us define
  \begin{equation}
    \label{404}
    \begin{split}
      &p_0 \coloneqq 2 \quad \text{and} \quad p_{i+1} \coloneqq (p_{i}
        -1)(d-1)+\tfrac1{d-1} \quad \text{for } i\geq 0,    \\
      &q_i \coloneqq (p_i -1)(d-1)^2+1 \quad \text{for } i\geq 0.
    \end{split}
  \end{equation}
  We  prove by induction that $g_{h,\delta_h} \in L^{p_i}(\nu)$ for all
  $i\geq 0$ by using the hypercontractivity estimate \eqref{405}. For
  $i=0$ we have $p_0=2$ and hence $g_{h,\delta_h} \in L^{p_0}(\nu)$ as
  seen in Corollary~\ref{5432}. Now assume
  $g_{h,\delta_h} \in L^{p_i}(\nu)$ for $i\geq 0$. Since
  $g_{h,\delta_h}=  R_{h,\delta_h} \,  g_{h,\delta_h}$ by
  Corollary~\ref{5432}, it follows that to prove
  $g_{h,\delta_h} \in L^{p_{i+1}}(\nu)$ it is enough to show
  $\hat g_{h,\delta_h}^{d-1} \in L^{p_{i+1}}(\nu)$, where we  abbreviated
  $\hat g_{h,\delta_h} (a) \coloneqq \mathbb E^Y \big[ g_{h,\delta_h} (\tfrac{a}{d-1} + Y ) \big]$.
  And indeed we have, using \eqref{404}, \eqref{405} and the induction
  hypothesis, that
  $\| \hat g_{h,\delta_h}^{d-1}\|_{L^{p_{i+1}}(\nu)}
  =\| \hat g_{h,\delta_h} \|^{d-1}_{L^{q_i}(\nu)}
  \leq \|g_{h,\delta_h}\|^{d-1}_{L^{p_i}(\nu)}  <\infty$.
  Next we show that the sequence $(p_i)_{i\geq 0}$ diverges to infinity
  as $i$ tends to infinity. To see this, note that
  $r_{i} \coloneqq \tfrac{d-2}{d-1}(d-1)^i  + \tfrac1{d-1} + 1$, $i\geq 0$,
  solves the recursion for $(p_i)_{i\geq 0}$ given in \eqref{404} and
  clearly $r_i \xrightarrow{i\to \infty} \infty$. This implies that
  $g_{h,\delta_h} \in L^q(\nu)$ for all $q\geq 1$ since we can take
  $i\geq 0$ such that $q < p_i$. Then
  $g_{h,\delta_h} \in L^{p_i}(\nu) \subseteq L^q(\nu)$.

  We turn to show \eqref{409}. Let $\zeta>0$ and take $q \geq 1$ large
  enough such that $2q \geq \tfrac1\zeta$. Since
  $g_{h,\delta_h}\in L^q(\nu)$ as just shown and
  $\nu = \mathcal N(0,\tfrac{d-1}{d-2})$ from above \eqref{e:defLh}, one
  has
  $\lim_{a\to \infty} g_{h,\delta_h}(a)^q \exp(-\tfrac{(d-2)a^2}{(d-1)2}) = 0$,
  which implies
  $\lim_{a\to \infty} g_{h,\delta_h}(a) \exp(-\tfrac{a^2}{2q}) = 0$. But
  this shows the statement of the lemma by the choice of $q$ (use that
    $g_{h,\delta_h}$ is continuous on $[h,\infty)$ by Corollary \ref{5432}).
\end{proof}

The estimate obtained in the previous lemma can be used to derive the
following recursive bound on $g_{h,\delta_h}$ which is the final
ingredient for the proof of \eqref{415}

\begin{lemma}
  \label{5678}
  Let $h>h_\star$ and consider the function $g_{h,\delta_h} \in L^2(\nu)$
  from Corollary \ref{5432}. For all $\eta>0$ there exists $c_{h,\eta}>0$
  such that
  \begin{equation}
    \label{4111}
    g_{h,\delta_h}(a) \leq (1+2\delta_h)g_{h,\delta_h}
    \big(\tfrac{a}{d-1}(1+ \eta) \big)^{d-1}
    \quad \text{ for all } a\geq c_{h,\eta}.
  \end{equation}
\end{lemma}

\begin{proof}
  Let $\eta>0$.  Because
  $g_{h,\delta_h}= R_{h,\delta_h} \,  g_{h,\delta_h}$ as obtained in
  \eqref{e:gisfixedpoint}, one has  for $a\geq h$ and with
  $Y\sim \mathcal N(0,\frac{d}{d-1})$ that
  \begin{equation}
    \label{407}
    \begin{split}
      g_{h,\delta_h}(a)
      & = (1+\delta_h) \Big(  \mathbb E^Y\Big[g_{h,\delta_h}(\tfrac{a}{d-1}+ Y)
          \bbone_{\{Y < \tfrac{\eta a}{d-1} \}} \Big] + \mathbb E^Y
        \Big[g_{h,\delta_h}(\tfrac{a}{d-1}+ Y) \bbone_{\{Y \geq  \tfrac{\eta a}{d-1}
              \}} \Big]    \Big)^{d-1}
      \\ & \leq (1+\delta_h)  \Big(   g_{h,\delta_h} \big(\tfrac{a}{d-1}(1+ \eta) \big) +
        \mathbb E^Y \Big[g_{h,\delta_h}(\tfrac{a}{d-1}+ Y) \bbone_{\{Y \geq
              \tfrac{\eta a}{d-1} \}} \Big]    \Big)^{d-1},
    \end{split}
  \end{equation}
  where in the last step we use that $g_{h,\delta_h}$ is a non-decreasing
  function (see \eqref{573} and \eqref{e:defPa}). Because
  $g_{h,\delta_h}\geq 1$, we further obtain from \eqref{407}  that for
  $a\geq h$
  \begin{equation}
    \label{412}
    g_{h,\delta_h}(a) \leq (1+\delta_h)  \Big( 1 +  \mathbb E^Y
      \Big[g_{h,\delta_h}(\tfrac{a}{d-1}+ Y) \bbone_{\{Y \geq \eta \tfrac{a}{d-1}
            \}} \Big]    \Big)^{d-1}   g_{h,\delta_h} \big(\tfrac{a}{d-1}(1+ \eta) \big)
    ^{d-1}.
  \end{equation}
  To bound the expectation on the right hand side of \eqref{412} we will
  apply \eqref{409} for some $\zeta>0$ depending on $\eta$. Choose
  $0<\zeta<\tfrac{d-1}{2d}\tfrac{\eta^2}{(1+\eta)^2}$, so that in
  particular $\zeta < \tfrac{d-1}{2d}$ and
  $\zeta(1+\eta)^2-\tfrac{d-1}{2d}\eta^2 < 0$. Then
  $z\mapsto\zeta(1+z)^2-\tfrac{d-1}{2d}z^2
  = (\zeta-\tfrac{d-1}{2d})z^2+2\zeta z + \zeta$
  is a parabola with negative leading coefficient and two zeros of
  opposite sign. As the parabola is negative at $z=\eta>0$, this shows
  that $\zeta(1+z)^2-\tfrac{d-1}{2d}z^2 < 0$ for all $z\geq \eta$. Hence
  there exists $c_{\zeta, \eta}=c_\eta'>0$ such that
  \begin{equation}
    \label{408}
    \zeta(1+z)^2-\tfrac{d-1}{2d}z^2 < -c_\eta' z \quad \text{for all $z\geq \eta$}.
  \end{equation}
  Since $Z\coloneqq \tfrac{d-1}{a}Y$ satisfies
  $Z \sim \mathcal N(0,\sigma_a^2)$ for
  $\sigma_a^2 \coloneqq \tfrac{(d-1)^2}{a^2} \tfrac{d}{d-1}$, we get the
  following bound for $a\geq (d-1)h$
  \begin{equation*}
    \begin{split}
      &\mathbb E^Y \Big[g_{h,\delta_h}(\tfrac{a}{d-1}+ Y) \bbone_{\{Y \geq \eta
            \tfrac{a}{d-1} \}} \Big]
      \ = \ \frac{1}{\sqrt{2\pi \sigma_a^2}} \int_{\eta}^\infty
      g_{h,\delta_h}\big(\tfrac{a}{d-1} (1+z)\big)
      \exp\big(-\tfrac{z^2}{2\sigma_a^2}\big)  \, dz \\
      &\overset{\eqref{409}}{\leq} \frac{c_{h,\eta}}{\sqrt{2\pi \sigma_a^2}}
      \int_{\eta}^\infty   \exp\Big((\tfrac{a}{d-1})^2 \big(\zeta
          (1+z)^2-\tfrac{d-1}{2d} z^2\big)\Big)  \, dz  \\
      &\overset{\eqref{408}}{\leq}  \frac{c_{h,\eta}}{\sqrt{2\pi \sigma_a^2}}
      \int_{\eta}^\infty   \exp\big(- (\tfrac{a}{d-1})^2 c_\eta'  z \big)  \,
      dz
      \ = \  \frac{c_{h,\eta}}{\sqrt{2\pi \sigma_a^2}} \frac{(d-1)^2}{a^2 c_\eta'}
      \exp\big(-(\tfrac{a}{d-1})^2 c_\eta' \eta \big),
    \end{split}
  \end{equation*}
  which tends to zero as $a$ tends to infinity. Hence there is
  $c_{h,\eta} >0$ such that
  $(1+\delta_h)  \big(  1 +
    \mathbb E^Y \big[g_{h,\delta_h}(\tfrac{a}{d-1}+ Y)
      \bbone_{\{Y \geq \eta \tfrac{a}{d-1} \}} \big]    \big)^{d-1}
  \leq (1+2\delta_h)$
  for all $a\geq c_{h,\eta}$. This, together with \eqref{412}, concludes
  the proof of Lemma \ref{5678}.
\end{proof}

The next and final lemma shows \eqref{415} and hence concludes the proof
of Theorem \ref{t:exponentialboundong}.

\begin{lemma}
  \label{9876}
  Let $h>h_\star$ and consider the function $g_{h,\delta_h} \in L^2(\nu)$
  from Corollary \ref{5432}. For all $\gamma > 0$ there exist
  $c_{h,\gamma}>0$ and $c_{h,\gamma}'>0$ such that
  $g_{h,\delta_h}(a) \leq c_{h,\gamma} \exp(c_{h,\gamma}' a^{1+\gamma})$
  for all $a\geq h$.
\end{lemma}

\begin{proof}
  Let $\gamma>0$ and take $\eta>0$ such that
  $1+\gamma = \log_{\frac{d-1}{1+\eta}}(d-1)$,  in particular
  $\tfrac{d-1}{1+\eta} > 1$. We abbreviate $K \coloneqq c_{h,\eta}$ for
  the constant from \eqref{4111}. Since $g_{h,\delta_h}$ is continuous on
  $[h,\infty)$ by Corollary~\ref{5432}, it is enough to find the
  requested bound on $g_{h,\delta_h}$ for all $a\geq K$. Define the
  intervals
  \begin{equation*}
    J_k \coloneqq \Big[ K \big( \tfrac{d-1}{1+\eta} \big)^k ,  K \big(
        \tfrac{d-1}{1+\eta} \big)^{k+1} \Big) \quad \text{for all $k\geq 0$},
  \end{equation*}
  which form a disjoint decomposition of $[K,\infty)$. For $a \geq K$ let
  $k(a)\geq 0$ be the unique $k\geq 0$ with $a \in J_k$, that is,
  $k(a) \coloneqq \lfloor \log_{\frac{d-1}{1+\eta}}(\tfrac{a}{K})\rfloor$.
  For such $a$ one can apply  \eqref{4111} iteratively $k(a)$ times to
  obtain
  \begin{equation}
    \label{413}
    \begin{split}
      g_{h,\delta_h}(a)& \leq (1+2\delta_h)g_{h,\delta_h} \big(\tfrac{a}{d-1}(1+ \eta)
        \big)^{d-1}
      \\&\leq (1+2\delta_h)^{1+(d-1)} g_{h,\delta_h} \Big(\tfrac{a}{(d-1)^2}(1+ \eta)^2
        \Big)^{(d-1)^2}                                                          \\
      &  \leq \ldots \leq (1+2\delta_h)^{\sum_{i=0}^{k(a)-1} (d-1)^i} g_{h,\delta_h}
      \Big(\underbrace{\tfrac{a}{(d-1)^{k(a)}}(1+ \eta)^{k(a)}}_{\in J_0}
        \Big)^{(d-1)^{k(a)}}                                                 \\
      & \leq \Big( (1+2\delta_h)\sup_{b\in J_0} g_{h,\delta_h}(b)  \Big)^{(d-1)^{k(a)}}.
    \end{split}
  \end{equation}
  Note that
  $(d-1)^{k(a)} \leq (d-1)^{\log_{\frac{d-1}{1+\eta}}\big(\tfrac{a}{K}\big)} =
  \big( \tfrac{a}{K} \big)^{\log_{\frac{d-1}{1+\eta}}(d-1)} =  \big( \tfrac{a}{K}
    \big)^{1+\gamma}$
  and therefore \eqref{413} implies that
  \begin{equation*}
    \begin{split}
      g_{h,\delta_h}(a) &\leq \exp \Big(  \big( \tfrac{a}{K} \big)^{1+\gamma} \ln \Big(
          (1+2\delta_h)\sup_{b\in J_0} g_{h,\delta_h}(b)  \Big) \Big)
      \leq c_{h,\gamma} \exp(c_{h,\gamma}' a^{1+\gamma})  \quad \text{for } a\geq K.
    \end{split}
  \end{equation*}
  As explained above, this proves the lemma.
\end{proof}

We end with some concluding remarks.  One might naturally wonder what can
be said about level-set percolation of the Gaussian free field on
$\treegraph$ near criticality. For example: can the result from Theorem
\ref{t:etacontinuity} be extended to $h_\star$, i.e. are the functions
$\eta$ and $\eta^+$ continuous or not at $h_\star$? Or also: does the
equality  \eqref{222} hold for $h=h_\star$, too?

Independently from that, and as remarked in the introduction, we apply a
number of the results obtained here in the accompanying paper \cite{AC2}
to establish a phase transition for level-set percolation of the
zero-average Gaussian free field on a class of finite regular expanders.



\bibliographystyle{alpha}
\bibliography{bibliographyfile}

\end{document}